\documentclass[a4paper,reqno,12pt]{amsart} 
\usepackage{amsmath} 

\usepackage{amssymb}      

\usepackage{yhmath}
\usepackage{mathdots}

\usepackage{amsthm}      

\usepackage{tikz,amsmath}
\usetikzlibrary{arrows}

\usepackage{epsfig}      

\usepackage{graphicx}
\usepackage[normalem]{ulem}
\usepackage[all,line,
]{xy} 
\CompileMatrices 

\newcommand{\Ad}{\operatorname{Ad}}

\newcommand{\id}{\operatorname{id}}

\newcommand{\Tr}{\operatorname{Tr}}

 \newcommand{\supp}{\operatorname{supp}}

   \theoremstyle{plain}
   \newtheorem{thm}{Theorem}[section]
   \newtheorem{prop}[thm]{Proposition}
   \newtheorem{lemma}[thm]{Lemma}  
   \newtheorem{cor}[thm]{Corollary}
   \theoremstyle{definition}
   
   \newtheorem{defn}[thm]{Definition}
   \newtheorem{example}[thm]{Example}
   \theoremstyle{remark}
   
   \newtheorem{remark}[thm]{Remark}

   \numberwithin{equation}{section}

        \date{\today}

\title[Homeomorphisms and KMS states]{KMS states on the crossed product $C^{*}$-algebra of a homeomorphism}
\author{Johannes Christensen and Klaus Thomsen}

\date{\today}

\email{ johannes@math.au.dk, matkt@math.au.dk}
\address{Department of Mathematics, Aarhus University, Ny Munkegade, 8000 Aarhus C, Denmark}

\begin{document}

\maketitle

\begin{abstract}
Let $\phi:X\to X$ be a homeomorphism of a compact metric space $X$. For any continuous function $F:X\to \mathbb{R}$ there is a one-parameter group $\alpha^{F}$ of automorphisms on the crossed product $C^*$-algebra $C(X)\rtimes_{\phi}\mathbb{Z}$ defined such that $\alpha^{F}_{t}(fU)=fUe^{-itF}$ when $f \in C(X)$ and $U$ is the canonical unitary in the construction of the crossed product. In this paper we study the KMS states for these flows by developing an intimate relation to the ergodic theory of non-singular transformations and show that the structure of KMS-states can be very rich and complicated. Our results are complete concerning the set of possible inverse temperatures; in particular, we show that when $C(X) \rtimes_{\phi} \mathbb Z$ is simple this set is either $\{0\}$ or the whole line $\mathbb R$. 
\end{abstract}

\section{Introduction}
Many $C^*$-algebras possess canonical groups of automorphisms arising from their construction and often they give rise to continuous one-parameter groups of automorphisms, in the following referred to as a flow. In certain models in quantum statistical mechanics such a flow represents the time-evolution while the self-adjoint elements of the $C^*$-algebra represent observables and the equilibrium states are given by states on the $C^*$-algebra which satisfy the KMS-condition, cf. \cite{BR}. For this and other reasons the KMS states of the flows that arise from the construction of various $C^*$-algebras have been investigated and in many cases completely described. Nonetheless the $C^*$-algebra constructed from a homeomorphism of a compact metric space has been missing in this picture despite that it is the most classical way to build an infinite dimensional simple unital $C^*$-algebra. We try in this paper to demonstrate the size of this oversight. 

Specifically, we study the flows $\alpha^F$ described in the abstract, given by a continuous function $F : X \to \mathbb R$ which we shall often refer to as a potential. Fundamental properties of $\alpha^F$ such as innerness and approximate innerness, as well as the factor types of its KMS states, depend heavily on the properties of $F$, and in particular of whether or not there is a solution to the cohomological equation
\begin{equation}\label{23-10-19}
F \ = \ h \circ \phi - h \ .
\end{equation}
The flow $\alpha^F$ is inner if and only if there is a continuous solution $h$ as we show in Theorem \ref{08-03-19c}, and it is approximately inner if and only if the equation admits a continuous solution in an asymptotic sense, which happens if and only if
\begin{equation*}
\int_X F \ \mathrm{d} \nu \ = \ 0 
\end{equation*}
for all $\phi$-invariant Borel probability measures $\nu$, cf. Theorem \ref{thm41}. The $\beta$-KMS states of $\alpha^F$ are almost completely determined by Borel probability measures $m$ that are $e^{\beta F}$-conformal in the sense that
$$
\frac{\mathrm{d} m \circ \phi}{\mathrm{d} m}  = e^{\beta F} \ ,
$$
cf. Lemma \ref{lemma21}; in fact, when $C(X)\rtimes_{\phi} \mathbb Z$ is simple there is an affine homeomorphism between the simplex of $\beta$-KMS states $\omega$ for $\alpha^F$ and the compact convex set of $e^{\beta F}$-conformal measures $m$ given by the equation
\begin{equation*}\label{23-10-19b}
\omega(a) \ = \ \int_X E(a) \ \mathrm{d} m \ ,
\end{equation*}
where $E : C(X) \rtimes_{\phi} \mathbb Z \to C(X)$ is the canonical conditional expectation, cf. Corollary \ref{01-08-19e}. Other $\beta$-KMS states exist only when there are $\phi$-periodic points $x\in X$, say of minimal period $p$, that are $F$-cyclic in the sense that $\sum_{i=1}^p F(\phi^i(x)) = 0$. In this case there is an $e^{\beta F}$-conformal measure $m$ concentrated on the $\phi$-orbit of $x$, and there is a closed face of $\beta$-KMS states, affinely homeomorphic to the simplex of Borel probability measures on the circle, that are all given by $m$ when restricted to $C(X)$, cf. Lemma \ref{01-08-19}. This kind of KMS states are easy to understand and they will be ignored in the rest of this introduction.

To distinguish the various types of KMS-states we study the factor types of the extremal $\beta$-KMS states, i.e. we investigate when the von Neumann algebra factor generated by the GNS representation of an extremal $\beta$-KMS state is of type $I,II_1,II_{\infty}$ or type $III$. Again this depends crucially on the solutions or lack of solutions to the cohomological equation \eqref{23-10-19}. Specifically, when $\omega$ is an extremal $\beta$-KMS state the corresponding $e^{\beta F}$-conformal measure $m$ is ergodic and the von Neumann algebra factor is isomorphic to the von Neumann algebra crossed product $L^{\infty}(m) \rtimes_{\phi} \mathbb Z$, cf. Lemma \ref{06-09-19}, which is of type $III$ if and only if there are no Borel function $h$ for which \eqref{23-10-19} holds $m$-almost everywhere, cf. Section \ref{factor} and Theorem \ref{15-08-19g}. In particular, the factor generated by $\omega$ is semi-finite if and only if there is a Borel function $h$ solving \eqref{23-10-19} $m$-almost everywhere. In this case there is $\sigma$-finite $\phi$-invariant Borel measure $\nu$ such that
\begin{equation*}\label{23-10-19c}
\mathrm{d} \nu \ = \ e^{-h} \mathrm{d} m \ ,
\end{equation*}
cf. Proposition \ref{propadd4}. Consequently $L^{\infty}(m) \rtimes_{\phi} \mathbb Z$ is finite, i.e. of type $I_n$ or type $II_1$, if and only if $e^{-h} \in L^1(m)$, and of type $I_{\infty}$ or $II_{\infty}$ otherwise. Finally, it is of type $I$ if $m$ is atomic and of type $II$ if not.

Having established a sound understanding of the rather close connection to the ergodic theory of non-singular transformations we go on to exploit this to show that all factor types occur. In fact, for any given type other than $I_n$, there is an irrational rotation of the circle and a choice of potential $F : \mathbb T \to \mathbb R$ such that the flow $\alpha^F$ on the corresponding irrational rotation algebra has a $1$-KMS states of the given type, cf. Section \ref{types}. The type $II_1$ case is easy to realize and occurs for every irrational rotation, cf. Proposition \ref{03-10-19}, while all the other cases require substantial work. Concerning the type $II_{\infty}$ and type $III$ cases this work was done by Katznelson in \cite{K} and we just translate his results. The type $I_{\infty}$ case is also difficult, but for certain irrational rotations the existence follows from work of Douady and Yoccoz, \cite{DY}. We show that there can be arbitrarily many extremal $\beta$-KMS states of type $I_{\infty}$ for any non-periodic homeomorphism provided the potential is chosen carefully, cf. Theorem \ref{thm73}. In particular, the restriction to certain irrational numbers in \cite{DY} is not necessary.

We consider next the question about the variation with $\beta$ of the simplex of $\beta$-KMS states. The connection to the ergodic theory of non-singular transformations is less helpful for this issue, but the work by Douady and Yoccoz does explicitly confront it in relation to diffeomorphisms $\phi$ of the circle when the potential is $F = \log D(\phi)$, \cite{DY}. In particular, they find in this case that there is one and only one $D(\phi)^{\beta}$-conformal measure for each $\beta \in \mathbb R$ when the total variation of $\log D(\phi)$ is finite. The uniqueness part of their proof can be adapted to show that for any irrational rotation and for any potential $F$ with finite total variation on the circle there is at most one $e^{\beta F}$-conformal measure for all $\beta \in \mathbb R$. One of our main results, Theorem \ref{18-10-19b}, shows that when $\phi$ is uniquely ergodic, the existence of any $e^{\beta F}$-conformal measure for any potential $F$ and any $\beta \neq 0$ is equivalent to the integral of $F$ with respect to the $\phi$-invariant Borel probability measure being zero, in which case they exist for all $\beta \in \mathbb R$. It follows that for a potential on the circle with bounded total variation and zero integral with respect to Lebesgue measure, the corresponding flow $\alpha^F$ on any irrational rotation algebra has one and only one $\beta$-KMS state for all $\beta \in \mathbb R$. In general, with no restriction on the total variation of $F$, there can be arbitrarily many extremal $\beta$-KMS states in this case as follows from Theorem \ref{thm73}. To show that the factor types realized by the extremal $\beta$-KMS states can vary with $\beta$ we elaborate on an example considered by Baggett, Medina and Merrill in \cite{BMM} to show that for any irrational rotation $C^*$-algebra there is a potential $F$ such that the flow $\alpha^F$ has extremal $\beta$-KMS states of type $II_1$ for all $\beta \geq 0$ while they all are of type $II_{\infty}$ or type $III$ when $\beta < 0$, cf. Proposition \ref{prop74(2)}.

Aside from indicating by examples that the simplices of $\beta$-KMS states can vary wildly with $\beta$ we leave the question about the general variation completely open. As we hope will become clear, this is a difficult question involving highly non-trivial problems in ergodic theory. The most fundamental question on this issue, however, is for which $\beta \in \mathbb R$ there are any $\beta$-KMS states at all, and this we can answer completely thanks to one of the main results, Theorem \ref{thm31}, which gives necessary and sufficient conditions for the existence of an $e^{\beta F}$-conformal measure for an arbitrary homeomorphism $\phi$ and an arbitrary potential $F$. It turns out that the set of $\beta$'s for which there is an $e^{\beta F}$-conformal measure is always one of the sets $\{0\}, \ \mathbb R, \ [0, \infty)$ or $(-\infty,0]$. All four possibilities occur, but surprisingly this is not the case when $\phi$ is minimal. In that case the set is either $\{0\}$ or $\mathbb R$, cf. Theorem \ref{Hurra}. As a result, regardless of the potential $F$, when $C(X)\rtimes_{\phi} \mathbb Z$ is simple the set of $\beta$'s for which there is a $\beta$-KMS state for the flow $\alpha^F$ is either $\{0\}$ or $\mathbb R$. 

It should be clear by now that the paper depends in large parts on work focusing on dynamical systems, and in particular the theory of non-singular transformations, and fortunately we are able to partly return the favour. By using one of our main results we can answer a question raised by Douady and Yoccoz in \cite{DY}, cf. Remark \ref{DYanswer}, and partly also a question raised by Schmidt, \cite{Sc}, cf. Remark \ref{Schmidt}. 

\smallskip

   \emph{Acknowledgement} The work was supported by the DFF-Research Project 2 `Automorphisms and Invariants of Operator Algebras', no. 7014-00145B.

\section{The flows}

Let $X$ be a compact metric space, or equivalently a second countable compact Hausdorf space, and let $\phi:X\to X$ be a homeomorphism of $X$. The crossed product $C^*$-algebra $C(X) \rtimes_{\phi} \mathbb Z$ of $(X,\phi)$ is the universal $C^*$-algebra generated by a copy of $C(X)$ and a unitary $U$ such that
\begin{equation*}
UfU^{*}=f\circ \phi^{-1} \ \text{ for all } f\in C(X) \ . 
\end{equation*} 
Given a real-valued continuous function $F : X \to \mathbb R$ there is a continuous one-parameter group $\alpha^{F}=\{\alpha^{F}_{t}\}_{t\in \mathbb{R}}$ of automorphisms of $C(X) \rtimes_{\phi} \mathbb Z$ determined by the condition that
\begin{equation*}
\alpha^{F}_{t}(fU) \ = \ f Ue^{-itF} \ = \ e^{-itF \circ \phi^{-1}} fU  
\end{equation*} 
for all $f \in C(X)$. Flows of this kind are characterized by the property that they contain $C(X)$ in the fixed point algebra, at least when $C(X)\rtimes_{\phi} \mathbb Z$ is simple; this follows from Proposition 5.2 in \cite{CT}. Until further notice we fix $X,\phi$ and $F$ and study here the innerness and approximate innerness of the flow $\alpha^F$.

\subsection{Innerness}

Recall that a flow $\alpha$ on a unital $C^*$-algebra $B$ is \emph{inner} when there is a self-adjoint element $h \in B$ such that
$$
\alpha_t(a) \ = e^{it h}ae^{-ith}
$$
for all $a \in B$ and $t \in \mathbb R$. Recall also that there is a conditional expectation $E : C(X)\rtimes_{\phi} \mathbb Z \ \to \ C(X)$ such that
$E(fU^n) \ = \ 0$ for $f \in C(X)$ when $n \neq 0$.

\begin{thm}\label{08-03-19c} The following are equivalent:
\begin{itemize}
\item[1)] The flow $\alpha^F$ is inner.
\item[2)] There exists a real-valued function $h\in C(X)$ such that $\alpha^F_{t}=\Ad e^{ith}$ for all $t\in \mathbb{R}$.
\item[3)] There is a continuous function $h : X \to \mathbb R$ such that $F =  h \circ \phi -h$. 
\end{itemize} 
\end{thm}
\begin{proof} 
1) $\Rightarrow$ 3): Let $h \in C(X)\rtimes_{\phi} \mathbb Z$ be a self-adjoint element such that $\alpha^F_t = \Ad e^{-ith}$. Then
\begin{equation*}
e^{-ith}Ue^{ith}=\alpha^{F}_{t}(U)=Ue^{-itF} 
\end{equation*}
and differentiation at $t =0$ shows that
$$
-ihU  + i Uh = - i UF\ ,
$$
which implies that $F \ = \ U^*hU - h$. By checking on elements of the form $fU^{n}$ we find that $E(U^{*}xU)=E(x)\circ\phi$. Hence $F=E(h)\circ\phi-E(h)$. 

3) $\Rightarrow$ 2) follows from observing that when $h\in C(X)$ is real-valued and $h\circ \phi-h=F$ then $\alpha^F_t = \Ad e^{-ith}$.
This completes the proof because the implication 2) $\Rightarrow$ 1) is trivial.
\end{proof}

When $\phi$ is minimal, in the sense that all $\phi$-orbits are dense in $X$, the three conditions in Theorem \ref{08-03-19c} are equivalent to the following.
\begin{itemize}
\item[4)] There is an element $x \in X$ such that 
$$
\sup_{n \in \mathbb N} \left| \sum_{k=0}^nF\left(\phi^k(x)\right)\right| \ < \ \infty \ .
$$
\end{itemize} 
This follows from a famous result of Gottschalk and Hedlund, \cite{GH}. 

\subsection{Approximate innerness}

Recall that a flow $\alpha = \left(\alpha_t\right)_{t \in \mathbb R}$ on a unital $C^*$-algebra $B$ is \emph{approximately inner} when there is a sequence $\left\{\alpha^n\right\}_{n=1}^{\infty}$ of inner flows on $B$ such that $\lim_{n \to \infty} \left\|\alpha_t(a) - \alpha^n_t(a)\right\| = 0$ uniformly on compact subsets of $\mathbb R$ for all $a \in B$, cf. \cite{PS}.

\begin{thm}\label{thm41} The following are equivalent:
\begin{itemize}
\item[1)] The flow $\alpha^F$ is approximately inner.

\item[2)] $\int_X F \ \mathrm{d}\nu = 0$ for all $\phi$-invariant Borel probability measures $\nu$.

\item[3)] There is a sequence $\{h_n\}$ of continuous real-valued functions on $X$ such that $\lim_{n \to \infty} h_n \circ \phi - h_n = F$ uniformly on $X$.
\end{itemize}
\end{thm}

\begin{proof} $1) \Rightarrow 2)$: Let $\{H_n\}$ be a sequence of self-adjoint elements of $C(X) \rtimes_{\phi} \mathbb Z$ such that $\lim_{n \to \infty} e^{itH_n}ae^{-itH_n} = \alpha^F_t(a)$ uniformly on compact subsets of $\mathbb R$ for all $a \in C(X) \rtimes_{\phi} \mathbb Z$. Let $\nu$ be a $\phi$-invariant Borel probability measure on $X$ and let $\tau$ be the corresponding tracial state on $C(X) \rtimes_{\phi} \mathbb Z$, i.e. for $f \in C(X)$ and $z \in \mathbb Z$ we have that
\begin{equation*}
\tau\left( fU^z\right) = \begin{cases} \int_X f \ \mathrm{d}\nu \ , & \ z = 0 \\ 0 \ ,  & \ z \neq 0 \ . \end{cases} 
\end{equation*}
We can choose $N \in \mathbb N$ so big that
\begin{equation*}\label{18-03-19}
\left\| e^{-iFt} -  U^*e^{itH_n}Ue^{-itH_n}\right\| \ = \ \left\| \alpha^F_t(U) -   e^{itH_n}Ue^{-itH_n}\right\| \  < \ 1
\end{equation*}
when $|t| \leq 1$ and $n \geq N$. Fix $t \in [0,1]$ and let $B$ denote the $C^*$-algebra of continuous function $f : [0,1] \to C(X) \rtimes_{\phi} \mathbb Z$ with the property that $f(0) \in \mathbb C 1$. Let $n \geq N$. Define $\overline{F}, \overline{H_n}$ and $\overline{K_n}$ in $B$ such that
\begin{itemize}
\item $\overline{F}(s) = stF$,
\item $\overline{H_n}(s) = st H_n$, and
\item $\overline{K_n}(s) = st U^*H_nU$
\end{itemize}  
for $s \in [0,1]$. Then
$$
\left\| e^{-i \overline{F}} - e^{i \overline{K_n}}e^{-i \overline{H_n}} \right\|  \ = \ \sup_{0 \leq s \leq t} \left\| \alpha^F_s(U) -   e^{isH_n}Ue^{-isH_n}\right\| 
$$
and hence 
\begin{equation*}
\left\| e^{-i \overline{F}} - e^{i \overline{K_n}}e^{-i \overline{H_n}} \right\|  \ < \ 1  \ ,
\end{equation*}
for $n \geq N$ and 
$$
\lim_{n \to \infty} \left\| e^{-i \overline{F}} - e^{i \overline{K_n}}e^{-i \overline{H_n}} \right\|  \ = \ 0 \ .
$$
This implies that there are self-adjoint elements $A_n \in B$, $n \geq N$, such that $\lim_{n \to \infty} A_n = 0$ and 
\begin{equation*}
e^{-i \overline{F}} = e^{iA_n} e^{i \overline{K_n}}e^{-i \overline{H_n}}  \ 
\end{equation*}
for all $n \geq N$. Let $\overline{\tau}$ be the tracial state on $B$ defined such that
\begin{equation*}
\overline{\tau}(f) \ = \ \int_0^1 \tau(f(s)) \ \mathrm{d}s \ 
\end{equation*}
and let 
\begin{equation*}
\Delta_{\overline{\tau}} : \ U_0(B) \ \to \ \mathbb R/\overline{\tau}\left(K_0(B)\right) 
\end{equation*}
be the corresponding determinant as defined by delaHarpe-Skandalis in \cite{dHS}. Note that
\begin{align*}
&\Delta_{\overline{\tau}} \left(e^{i \overline{K_n}}\right) \ = \ \frac{1}{2\pi} \int_0^1 st \tau(U^*H_nU) \ \mathrm{d} s \\ 
& = \ \frac{1}{2\pi} \int_0^1 st \tau(H_n) \ \mathrm{d} s  \ = \ - \Delta_{\overline{\tau}} \left(e^{-i \overline{H_n}}\right) 
\end{align*}
modulo $\overline{\tau}\left(K_0(B)\right) $. Since $\Delta_{\overline{\tau}}$ is a homomorphism it follows that
\begin{align*}
& -\frac{t}{4\pi} \int_X F \ \mathrm{d} \nu \ = \  -\frac{1}{2\pi} \int_0^1 st \int_X F \ \mathrm{d} \nu \ \mathrm{d} s \\
& = \Delta_{\overline{\tau}}\left(e^{-i \overline{F}}\right)  \ = \Delta_{\overline{\tau}}\left(e^{i A_n}\right) = \frac{1}{2\pi} \overline{\tau}(A_n)
\end{align*}
modulo $\overline{\tau}\left(K_0(B)\right) $. Since $K_0(B) = \mathbb Z$ with the generator represented by the unit $1 \in B$ we have that $\overline{\tau}\left(K_0(B)\right) = \mathbb Z$. 
Thus
$$
 -\frac{t}{4\pi} \int_X F \ \mathrm{d} \nu \ = \ \frac{1}{2\pi} \overline{\tau}(A_n) 
 $$ modulo $\mathbb Z$. Since $\lim_{n \to \infty} \frac{1}{2\pi} \overline{\tau}(A_n)  = 0$, this implies that $ -\frac{t}{4\pi} \int_X F \ \mathrm{d} \nu \in \mathbb Z$. As this conclusion holds for all $t \in [0,1]$ it follows that $\int_X F \ \mathrm{d} \nu \ = \ 0$.


 $2) \Rightarrow 3)$: We claim that the sequence
\begin{equation*}
\frac{1}{n}\sum_{i=0}^{n-1} F \circ \phi^i \ , \ n = 1,2,3, \cdots 
\end{equation*}
converges uniformly to $0$. Assume not. There is then an $\epsilon > 0$ and sequences $N_1 < N_2 < N_3 < \cdots$ in $\mathbb N$ and $\{x_i\}_{i=1}^{\infty}$ in $X$ such that 
\begin{equation*}
\left|\frac{1}{N_k}\sum_{i=0}^{N_k-1} F \circ \phi^i(x_k)\right| \ \geq \ \epsilon 
\end{equation*}
for all $k$. Let $\delta_y$ denote the Dirac measure at $y$. For each $k$ we consider the measure
\begin{equation*}
\nu_k = \frac{1}{N_k}\sum_{i=0}^{N_k-1} \delta_{\phi^i(x_k)} \ .
\end{equation*}
Any weak* condensation point of the sequence $\{\nu_k\}_{k=1}^{\infty}$ in the compact set of Borel probability measures on $X$ is a $\phi$-invariant Borel probability measure $\nu$ such that $\left|\int_X F \ \mathrm{d}\nu\right| \geq \epsilon$, contradicting 2) and proving the claim. Set
\begin{equation*}
h_n =  -\frac{1}{n} \sum_{j=1}^n \sum_{i=0}^{j-1} F \circ \phi^i \ ,
\end{equation*}
and note that 
\begin{equation*}
h_n\circ \phi - h_n  \ = \ F \ - \ \frac{1}{n} \sum_{i=1}^{n} F \circ \phi^i \ .
\end{equation*}
It follows from the claim above that $\lim_{n \to \infty} h_n\circ \phi - h_n \ = \ F$ uniformly on $X$.

$3) \Rightarrow 1)$: This follows because
\begin{equation*}
\lim_{n \to \infty} e^{-ith_n} a e^{ith_n} \ - \ \alpha^F_t(a) \ = \ 0
\end{equation*}
uniformly on compact subsets of $\mathbb R$ for all $a \in C(X)\rtimes_{\phi} \mathbb Z$.

\end{proof}

\section{KMS states of $\alpha^F$ and conformal measures} \label{section2}


Let $\beta \in \mathbb R$. Following \cite{BR} a state $\omega$ on $C(X)\rtimes_{\phi}\mathbb{Z}$ is a $\beta$-KMS state for $\alpha^{F}$ when
\begin{equation*}
\omega(ab)=\omega(b\alpha_{i\beta}^{F}(a))
\end{equation*}
for all $a$ and $b$ in a dense $\alpha^{F}$-invariant $*$-algebra of $\alpha^{F}$-analytic elements in $C(X)\rtimes_{\phi} \mathbb{Z}$. To describe these KMS states we need the notion of a conformal measure with respect to $\phi$; a notion first introduced in dynamical systems by Patterson, \cite{P}, and Sullivan, \cite{S}. In a generality exceeding what we need here the notion was coined by Denker and Urbanski in \cite{DU}. For the present purposes it is convenient to define an \emph{$e^{\beta F}$-conformal measure} to be a Borel probability measure $m$ on $X$ with the property that
\begin{equation*} \label{eqconf}
m(\phi(B))=\int_{B}e^{\beta F(x)} \ \mathrm{d} m (x)
\end{equation*} 
for all Borel sets $B\subseteq X$. In other words, $m\circ \phi$ is absolutely continuous with respect to $m$ with Radon-Nikodym derivative $e^{\beta F}$, or
\begin{equation} \label{eqint}
\int_{X} f \ \mathrm{d}m=\int_{X} f\circ \phi \ e^{\beta F} \ \mathrm{d}m 
\end{equation}
for all $f\in C(X)$. We note that $m$ is $e^{\beta F}$-conformal for $\phi$ if and only if it is $e^{-\beta F \circ \phi^{-1}}$-conformal for $\phi^{-1}$.


\begin{lemma} \label{lemma21}
If $\omega$ is a $ \beta$-KMS state for $\alpha^{F}$ the restriction of $\omega$ to $C(X)$ defines an $e^{\beta F}$-conformal measure $m_{\omega}$ on $X$.
\end{lemma}
\begin{proof} For any $f\in C(X)$,
\begin{align*}
\omega(f) &= \omega(U (f\circ \phi) U^{*}) = \omega((f\circ \phi) U^{*}\alpha^{F}_{i\beta}(U))\\
& = \omega((f\circ \phi) U^{*}Ue^{\beta F}) = \omega((f\circ \phi) e^{\beta F}),
\end{align*}
which proves the lemma.
\end{proof}

\begin{lemma}\label{lemma22} The map $\omega \mapsto m_{\omega}$ of Lemma \ref{lemma21} from $\beta$-KMS states for $\alpha^F$ to the $e^{\beta F}$-conformal measures for $\phi$ is surjective. Specifically, when $m$ is an $e^{\beta F}$-conformal measure there is a $\beta$-KMS states $\omega_m$ for $\alpha^F$ such that
$$
\omega_m(a) = \int_X E(a) \ \mathrm{d} m \ 
$$
for all $a \in C(X)\rtimes_{\phi} \mathbb Z$.
\end{lemma}
\begin{proof} Note that $E(fU^{n}gU^{m})=0$ when $m\neq -n$ and that for $n>0$ then  
\begin{align*}
&\omega_{m}(fU^{n}gU^{-n})=\int_{X}f (g\circ \phi^{-n})\ \mathrm{d}m 
=\int_{X}(f\circ \phi^{n})g\exp(\beta \sum_{i=0}^{n-1} F\circ \phi^{i}) \ \mathrm{d}m \\
&=\omega_{m}( g U^{-n} f U^{n}\exp(\beta \sum_{i=0}^{n-1} F\circ \phi^{i})) 
=\omega_{m}( g U^{-n} \alpha_{i\beta}^{F}(f U^{n})) \ .
\end{align*}
Carrying out a similar calculation for $n<0$ it follows that $\omega_m$ is a $\beta$-KMS state. Since $m_{\omega_m} = m$ this shows that the map from Lemma \ref{lemma21} is surjective.
\end{proof}

In particular, if there is a $\beta$-KMS state for $\alpha^F$ there is also one which factorizes through $E$. When $C(X)\rtimes_{\phi} \mathbb Z$ is simple the flows we consider in this paper are characterized by the possible existence of a $\beta$-KMS state with this property for some $\beta \neq 0$, cf. Proposition 5.3 in \cite{CT}. Remark that setting $f=1$ in \eqref{eqint} gives us the following informative constraint on the behaviour of $F$ when there exists $\beta$-KMS states for $\beta \neq 0$.

\begin{cor}\label{cor35}  Assume that $F(x) > 0$ for all $x \in X$ or that $F(x) < 0$ for all $x \in X$. There are no $\beta$-KMS states for $\alpha^{F}$ when $\beta \neq 0$.
\end{cor}

In general the map $\omega \mapsto m_{\omega}$ is not injective and hence not all $\beta$-KMS states are diagonal in the sense of \cite{CT}, i.e. they do not all factorize through the conditional expectation $E$. We proceed with the identification of the $\beta$-KMS states that are not diagonal and obtain in this way also a necessary and sufficient condition for the injectivity of the map from Lemma \ref{lemma21}. For this purpose the following example will be useful.

\begin{example}\label{31-07-19a} Consider the case where $X$ is the finite set $X = \{1,2,\cdots, p\}$ and let $\sigma$ be the cyclic permutation:
$$
\sigma(j) = \begin{cases} j+1, \ & j \leq p-1 \ ,  \\ 1, & \ j = p \ . \end{cases}
$$
When $F : \{1,2,\cdots, p\} \to \mathbb R$ is a function there are clearly no $e^{\beta F}$-conformal measures for $\beta \neq 0$ unless 
\begin{equation}\label{31-07-19}
\sum_{i=1}^p F(i) = 0 \ ,
\end{equation}
and hence no $\beta$-KMS states for $\alpha^F$ on $C(\{1,2,\cdots, p\}) \rtimes_{\sigma} \mathbb Z$ when $\beta \neq 0$ unless \eqref{31-07-19} holds. Assume therefore that \eqref{31-07-19} holds. It is well-known that the crossed product $C(X) \rtimes_\sigma \mathbb Z $ is a copy of $C(\mathbb T) \otimes M_p(\mathbb C) =  C(\mathbb T , M_p(\mathbb C))$. To describe an isomorphism explicitly let $\{e_{i,j}\}_{i,j =1}^p$ be the standard matrix units in $M_p(\mathbb C)$. An isomorphism $\pi :C(X) \rtimes_\sigma \mathbb Z  \to C(\mathbb T , M_p(\mathbb C))$ is then obtained by setting
$$
\pi(f)(z) = \sum_{j=1}^p f(j)e_{j,j}
$$
when $f \in C(X)$ and 
$$
\pi(U)(z) \ = \  z e_{1,p} \ + \ \sum_{j=1}^{p-1} e_{j+1,j} 
$$
for all $z \in \mathbb T$. Since we assume that \eqref{31-07-19} holds the self-adjoint matrix
$$
H = - \sum_{j=2}^p \left(\sum_{k=1}^{j-1} F(k)\right)e_{j,j} \in M_p(\mathbb C)
$$
will have the property that $\pi \circ \alpha^F_t(a) = e^{itH}\pi(a)e^{-itH}$ for all $a \in C(X) \rtimes_\sigma \mathbb Z$. Thus the set of $\beta$-KMS states for $\alpha^F$ on $C(X) \rtimes_\sigma \mathbb Z$ is affinely homeomorphic via $\pi$ to the set of $\beta$-KMS states for $\id_{C(\mathbb T)} \otimes \Ad e^{itH}$ on $C(\mathbb T)\otimes M_p(\mathbb C)$ and the latter set is easily described: Let $\tau_{\beta}$ be the $\beta$-KMS state for $\Ad e^{itH}$ on $M_p(\mathbb C)$ given by
$$
\tau_{\beta}(x)  = \frac{\Tr  \left(e^{-\beta H}x\right)}{\Tr  \left(e^{-\beta H}\right)}  \ .
$$
For every Borel probability measure $m$ on $\mathbb T$ we can define a $\beta$-KMS state $\omega_m$ on $C(\mathbb T)\otimes M_p(\mathbb C)$ given on simple tensors by the formula
$$
\omega_m(f \otimes x) \ = \ \tau_{\beta}(x) \int_{\mathbb T} f \ \mathrm{d}m  \ .
$$
It is then straightfoward to see that the map $m \mapsto \omega_m \circ \pi$ gives an affine homeomorphism between the  simplex of Borel probability measures on $\mathbb T$ and the simplex of $\beta$-KMS states for $C(X)\rtimes_{\sigma} \mathbb Z$.

\end{example}

We return to the general case.

\begin{lemma}\label{31-07-19c} Let $\omega$ be a $\beta$-KMS state for $\alpha^F$ and let $(H,\pi_{\omega},u)$ be the GNS representation of $\omega$. Then $\pi_{\omega} : C(X) \to B(H)$ extends to a $*$-homomorphism $\overline{\pi}_{\omega} : L^{\infty}(m_{\omega}) \to \pi_{\omega}\left( C(X) \rtimes_{\phi} \mathbb Z\right)''$ with the following properties:
\begin{itemize}
\item $\pi_{\omega}(U) \overline{\pi}_{\omega}(f)\pi_{\omega}(U^*) = \overline{\pi}_{\omega}(f \circ \phi^{-1}), \ f \in L^{\infty}(m_{\omega})$, 
\item $\left<\overline{\pi}_{\omega}(f) u,u\right> = \int_X f \ \mathrm{d}m_{\omega} , \ f \in L^{\infty}(m_{\omega})$, and
\item when $\{h_n\}$ is a uniformly norm-bounded sequence in $L^{\infty}(m_{\omega})$ converging $m_{\omega}$-almost everywhere to $h \in L^{\infty}(m_{\omega})$, then 
$$
\lim_{n\to \infty} \overline{\pi}_{\omega}(h_n) = \overline{\pi}_{\omega}(h)
$$ 
in the strong operator topology.
\end{itemize} 
\end{lemma}
\begin{proof}
Set $H_{0}:=\overline{\pi_{\omega}(C(X))u}$. The triple $(\pi_{\omega}|_{H_{0}}, H_{0}, u)$ is a GNS representation for $\omega|_{C(X)}$ and hence $(\pi_{\omega}|_{H_{0}}(C(X)))''$ is naturally isomorphic to $L^{\infty}(m_{\omega})$. Since $\omega$ is a KMS state the vector $u$ is separating for $\pi_{\omega}(C(X)\rtimes_{\phi}\mathbb{Z})''$, and hence the map $\pi_{\omega}(C(X))''\ni T \to T|_{H_{0}}\in (\pi_{\omega}|_{H_{0}}(C(X)))''$ is an isomorphism. Combining these two isomorphisms we get our desired $*$-homomorphism.
\end{proof}

\begin{lemma}\label{30-07-19d} Let $\omega$ be a $\beta$-KMS state for $\alpha^F$, $\beta  \neq 0$. There are unique numbers $t_{\infty} \geq 0$ and $t_p \geq 0, \ p =1,2,3, \cdots$  such that
\begin{equation}\label{30-07-19b}
\omega = t_{\infty}\omega_{\infty} + \sum_{p=1}^{\infty} t_p \omega_p \ ,
\end{equation}
where
\begin{itemize}
\item $t_{\infty} + \sum_{p=1}^{\infty} t_p \ = \ 1$,
\item when $t_{\infty}$ is positive, $\omega_{\infty}$ is a $\beta$-KMS state for $\alpha^F$ and $m_{\omega_{\infty}}$ is concentrated on the set of points that are not $\phi$-periodic, and
\item when $t_p$ is positive, $\omega_p$ is a $\beta$-KMS state for $\alpha^F$ and $m_{\omega_p}$ is concentrated on
\begin{equation}\label{30-07-19c}
\left\{ x \in X: \ \phi^p(x) = x, \ \sum_{j=0}^{p-1} F\left(\phi^j(x)\right) = 0, \ \phi^j(x) \neq x, 1\leq j < p \right\} \ .
\end{equation}
\end{itemize}   
\end{lemma}
\begin{proof} 
 Let $M$ denote the set of points in $X$ that are not $\phi$-periodic. Set $t_{\infty} = m_{\omega}(M)$ and $t_p = m_{\omega}(M_p)$ where 
 $$
 M_p = \left\{ x \in X: \ \phi^p(x) = x, \ \phi^j(x) \neq x, \ 1\leq j < p \right\} \ .
 $$
Then $t_{\infty} + \sum_{p=1}^{\infty} t_p \ = \ 1$. Assume that $t_{\infty} \neq 0$. Since $M$ is $\phi$-invariant it follows from Lemma \ref{31-07-19c} that $\overline{\pi}_{\omega}(1_M)$ is central in $\pi_{\omega}\left( C(X) \rtimes_{\phi} \mathbb Z\right)''$ and we define a state $\omega_{\infty}$ on $C(X)\rtimes_{\phi} \mathbb Z$ such that
$$
\omega_{\infty}(a) = t_{\infty}^{-1} \left< \overline{\pi}(1_M)\pi_{\omega}(a)u,u\right> \ .
$$ 
It follows from Corollary 5.3.4 in \cite{BR} that $\omega_{\infty}$ is a $\beta$-KMS state. Similarly, when $t_p \neq 0$ we define a $\beta$-KMS state $\omega_p$ on $C(X)\rtimes_{\phi} \mathbb Z$ such that
$$
\omega_{p}(a) = t_p^{-1}\left< \overline{\pi}(1_{M_p})\pi_{\omega}(a)u,u\right> \ .
$$ 
Note that
$$
m_{\omega_{\infty}}(B) = t_{\infty}^{-1}m_{\omega}(B \cap M)
$$
and 
$$
m_{\omega_p}(B) = t_{p}^{-1}m_{\omega}(B \cap M_p)
$$
for all Borel sets $B \subseteq X$. In particular, $m_{\omega_{\infty}}$ and $m_{\omega_p}$ are concentrated on $M$ and $M_p$, respectively. Finally, note that 
$$
m_{\omega_p}(B) = m_{\omega_p}(\phi^p(B)) = \int_B e^{\beta \sum_{j=0}^{p-1} F\circ \phi^j(x)} \ \mathrm{d} m_{\omega_p}(x)
$$
for all Borel sets $B \subseteq X$ since $m_{\omega_p}$ is $e^{\beta F}$-conformal and concentrated on $M_p$.  This implies that $\sum_{j=0}^{p-1} F\circ \phi^j(x) = 0$ for $m_{\omega_p}$-almost all $x$ since $\beta \neq 0$. It follows that $m_{\omega_p}$ is concentrated on the set \eqref{30-07-19c}.

\end{proof}

 Recall that a measure $m$ on $X$ is $\phi$-ergodic if $m(B)=0$ or $m(B^{C})=0$ when $B$ is a Borel set such that $\phi^{-1}(B)=B$.

\begin{lemma}\label{31-07-19b} Let $\omega$ be an extremal $\beta$-KMS state for $\alpha^F, \ \beta \neq 0$. Then $m_{\omega}$ is $\phi$-ergodic.
\end{lemma}
\begin{proof} Assume $m_{\omega}$ is not ergodic. There is then a $\phi$-invariant Borel set $A_1 \subseteq X$ such that $0 < m_{\omega}(A_1) < 1$. Set $A_2 = X \backslash A_1$. As in the proof of Lemma \ref{30-07-19d} we can then define $\beta$-KMS states $\omega_i$ such that
$$
\omega_{i}(a) = m_{\omega}(A_i)^{-1}\left< \overline{\pi}(1_{A_i})\pi_{\omega}(a)u,u\right> \ .
$$ 
Then $m_{\omega_i}$ is concentrated on $A_i$ and hence $\omega_1 \neq \omega_2$. Since $\omega = m_{\omega}(A_1)\omega_1 + m_{\omega}(A_2)\omega_2$ this contradicts the assumed extremality of $\omega$. 

\end{proof}

A finite $\phi$-orbit $\mathcal O$ will be called \emph{$F$-cyclic} when $\sum_{x \in \mathcal O} F(x) = 0$.

\begin{lemma}\label{30-07-19dd} Let $\omega$ be an extremal $\beta$-KMS state for $\alpha^F$, $\beta \neq 0$. Then $m_{\omega}$ is either concentrated on the set of non-periodic points or on a finite $F$-cyclic $\phi$-orbit.
\end{lemma}
\begin{proof} Assume that $m_{\omega}$ is not concentrated on the set of non-periodic points. It follows then from Lemma \ref{30-07-19d} that $m_{\omega}$ is concentrated on 
\begin{equation*}
M_p = \left\{ x \in X: \ \phi^p(x) = x, \ \sum_{j=1}^p F\left(\phi^j(x)\right) = 0, \ \phi^j(x) \neq x, \ 1 \leq j < p \right\} 
\end{equation*}
 for some $p \in \mathbb N$. Choose a metric for the topology of $X$ and let $D$ be the corresponding Hausdorff metric on the set of compact subsets of $X$. For $x \in M_p$ let $\mathcal O_x$ denote the orbit of $x$ under $\phi$. For each $k \in \mathbb N$ we can then construct a partition
 $$
 M_p = A^k_1 \sqcup A^k_2 \sqcup A^k_3 \sqcup \cdots \sqcup A^k_{n_k}
 $$
 of $M_p$ into $\phi$-invariant Borel sets $A^k_i$ such that 
 $$
 D(\mathcal O_x,\mathcal O_y) \leq \frac{1}{k}
 $$
 when $x,y \in A^k_i, \ i =1,2,\cdots, n_k$. Since $m_{\omega}$ is ergodic by Lemma \ref{31-07-19b} there is an $i_k \in \{1,2,\cdots, n_k\}$ such that $m_{\omega}(A^k_{i_k}) =1$. Then 
 $$
 m_{\omega}\left( \bigcap_{k=1}^{\infty} A^k_{i_k}\right) \ = \ 1 \ .
 $$
This completes the proof since $\bigcap_{k=1}^{\infty} A^k_{i_k}$ must be a single orbit.

\end{proof}

Let $ x \in X$ be a point of minimal $\phi$-period $p$. Define $F_x : \{1,2,\cdots, p\} \to \mathbb R$ such that $F_x(i) = F(\phi^i(x))$. Let $\pi_x : C(X)\rtimes_{\phi} \mathbb Z \to C(\{1,2,\cdots, p\})\rtimes_{\sigma} \mathbb Z$ be the $*$-homomorphism obtained from the equivariant $*$-homomorphism $C(X) \to C(\{1,2,\cdots, p\})$ sending $f$ to the function $\tilde{f} \in C(\{1,2,\cdots ,p\})$ where
$$
\tilde{f}(i) = f(\phi^i(x))\ .
$$
Note that $\alpha^{F_x}_t \circ \pi_x = \pi_x \circ \alpha^F_t$ . As observed in Example \ref{31-07-19a} we can identify $C(\{1,2,\cdots, p\})\rtimes_{\sigma} \mathbb Z$ with $C(\mathbb T) \otimes M_p(\mathbb C)$ and we consider therefore also $\pi_x$ as a $*$-homomorphism $\pi_x : C(X) \rtimes_{\phi} \mathbb Z \to C( \mathbb T) \otimes M_p(\mathbb C)$. As observed in Example \ref{31-07-19a} there are no $\beta$-KMS states for $\alpha^{F_x}$ unless $\sum_{i=1}^p F(\phi^i(x)) = 0$, i.e. the orbit of $x$ is $F$-cyclic.

\begin{lemma}\label{01-08-19}  Let $ \mathcal O$ be a finite $F$-cyclic $\phi$-orbit. The set $\mathcal F^{\mathcal O}$ of $\beta$-KMS states $\omega$ of $\alpha^F$ such that $m_{\omega}$ is concentrated on $\mathcal O$ is a closed face in the simplex of $\beta$-KMS states for $\alpha^F$, and $\mathcal F^{\mathcal O}$ is affinely homeomorphic to the simplex of Borel probability measures on $\mathbb T$.  
\end{lemma}
\begin{proof} Let $x \in \mathcal O$. We claim that $\omega$ factorises through $\pi_x$ if and only if $m_{\omega}$ is concentrated on  $\mathcal O$. To show this it suffices to show that $\omega$ factorises through $\pi_x$ when $m_{\omega}$ is concentrated on $\mathcal O$ since the reverse implication is trivial. To this end observe that since $\mathbb Z$ is an exact group the sequence
$$
0 \ \to \ C_0(X \backslash \mathcal O) \rtimes_{\phi} \mathbb Z \ \to \ C(X) \rtimes_{\phi} \mathbb Z \ \to \ C(\mathcal O)\rtimes_{\phi} \mathbb Z \ \to \ 0 
$$
is exact, implying that $\ker \pi_x$ is the closed linear span of elements of the form $fU^k$ where $f \in C_0(X \backslash \mathcal O)$ and $k \in \mathbb Z$. Since
\begin{align*}
& \left|\omega(fU^k)\right|^2 \leq \omega(f^*f) = \int_X |f|^2 \ \mathrm{d} m_{\omega} = 0 \ ,
\end{align*}
because $m_{\omega}$ is concentrated on $\mathcal O$, it follows that $\omega$ factorises through $\pi_x$, proving the claim. It follows that the set of $\beta$-KMS states $\omega$ for which $m_{\omega}$ is concentrated on $\mathcal O$ is a closed face since the set of those that factorise through $\pi_x$ clearly is. The remaining statements follow from Example \ref{31-07-19a}.
\end{proof}

\begin{thm}\label{01-08-19a} Let $\omega$ be an extremal $\beta$-KMS state, $\beta \neq 0$. Then either
$$
 \omega(a) = \int_X E(a) \ \mathrm{d} m_{\omega} \ \ \ \forall a \in C(X)\rtimes_{\phi} \mathbb Z \ ,
 $$
where $m_{\omega}$ is $\phi$-ergodic and concentrated on the set of points that are not $\phi$-periodic or $\omega \in \mathcal F^{\mathcal O}$ for some finite $F$-cyclic $\phi$-orbit $\mathcal O$. 
\end{thm}
\begin{proof} It remains to show that $\omega$ factorises through $E$ when $m_{\omega}$ is concentrated on the set $M$ of points not periodic under $\phi$. For this it suffices to show that $\omega(fU^n) =0$ when $f \in C(X)$ is non-negative and $n\neq 0$. Let $M_{n}\subseteq X$ be the closed set of $n$-periodic points, and for $\varepsilon >0$ find an open set $V$ with $M_{n}\subseteq V$ and $m(V)<\varepsilon$. Every $x\notin M_{n}$ has an open neighborhood $V_{x}$ such that $V_{x}\cap\phi^{-n}(V_{x})=\emptyset$. We can therefore write $f=f_{0}+\sum_{i>0} f_{i}$ as a finite sum of non-negative continuous functions with $\supp(f_{0})\subseteq V$ and $\supp(f_{i})\subseteq V_{i}$ where $V_{i}\cap\phi^{-n}(V_{i})=\emptyset$. For $i>0$ we find that
\begin{align*}
&\omega(f_{i}U^{n})= \omega\left(\sqrt{f_i}U^nU^{-n}\sqrt{f_i}U^n \right) = \omega(\sqrt{f_{i}}U^{n} \sqrt{f_{i}}\circ\phi^{n})\\
&
=\omega(U^{n}\sqrt{f_{i}}\circ\phi^{n}\sqrt{f_{i}} )=0 \ ,
\end{align*}
and hence $|\omega(fU^{n})|=|\omega(f_{0}U^{n})|\leq \sqrt{\omega(f_{0}^{2})} \leq \lVert f \rVert_{\infty} \varepsilon$. Since $\varepsilon > 0$ was arbitrary it follows that $\omega(fU^{n})=0$, as desired.
\end{proof} 

It is possible to base a proof of Theorem \ref{01-08-19a} on the work of Neshveyev and the first author, \cite{N} and \cite{Ch}, but the necessary translation from the groupoid picture would not make the proof any shorter. 

\begin{cor}\label{01-08-19d} Let $\beta  \in \mathbb R \backslash \{0\}$. The map $\omega \mapsto m_{\omega}$ of Lemma \ref{lemma21} from $\beta$-KMS states for $\alpha^F$ to the $e^{\beta F}$-conformal measures for $\phi$ is injective, and hence an affine homeomorphism if and only if there are no finite $F$-cyclic $\phi$-orbits.
\end{cor}


\begin{cor}\label{01-08-19e} Let $\beta  \in \mathbb R \backslash \{0\}$. Assume that $X$ is not a finite set and that $\phi$ is minimal. The map $\omega \mapsto m_{\omega}$ of Lemma \ref{lemma21} from $\beta$-KMS states for $\alpha^F$ to the $e^{\beta F}$-conformal measures for $\phi$ is an affine homeomorphism.
\end{cor}


It follows from Corollary \ref{01-08-19e} that the map from Lemma \ref{lemma21} is an affine homeomorphism when $\beta \neq 0$ and $C(X)\rtimes_{\phi} \mathbb Z$ is simple.

\section{Non-singular transformations and factor types}\label{factor}


In the following we will introduce the notion of types of conformal measures using the standard notion of type of an equivalence relation. We refer the reader to Definition 3.4 in \cite{FM} for the definition of the types $I_{n}$, $n\in \mathbb{N}$, $I_{\infty}$, $II_{1}$, $II_{\infty}$ and  $III $ of equivalence relations. For our homeomorphism $\phi$ on the compact metric space $X$ we consider the orbit-equivalence relation $\mathcal{R}$ where $x\sim y$ if and only if there exists $n\in \mathbb{Z}$ such that $\phi^{n}(x)=y$. It follows from Proposition 3.1 in \cite{FM} that for any $e^{\beta F}$-conformal measure $m$ the space $X$ can be decomposed into $\phi$-invariant Borel sets 
\begin{equation}\label{eqadd1}
X =\bigsqcup_{n\in \mathbb{N}} M_{I_{n}} \sqcup M_{I_{\infty}}\sqcup M_{II_{1}}\sqcup M_{II_{\infty}}\sqcup M_{III}
\end{equation}
unique up to measure zero such that $\mathcal{R}$ restricted to the set $M_{\alpha}$ is of type $\alpha$. Using this decomposition of the equivalence relation $\mathcal{R}$ we can define different types of conformal measures.

\begin{defn}\label{15-08-19} An $e^{\beta F}$-conformal measure $m$ is
\begin{itemize}
\item of type $I_p$ when $m(M_{I_p}) = 1$,
\item of type $I_{\infty}$ when $m(M_{I_{\infty}}) = 1$,
\item of type $II_1$ when $m(M_{II_1}) = 1$,
\item of type $II_{\infty}$ when $m(M_{II_{\infty}}) = 1$, and
\item of type $III$ when $m(M_{III}) = 1$.
\end{itemize}
\end{defn}

In the following we will describe the sets occurring in \eqref{eqadd1} in terms of $\beta F$ when $\beta \neq 0$. Fix therefore an $e^{\beta F}$-conformal measure $m$ and a decomposition of $X$ as in \eqref{eqadd1} with respect to $m$. To ease notation we set
$$
S_k(F) = \begin{cases} -\sum_{j=1}^{|k|} F\circ \phi^{-j} \ , & k \leq -1 \ , \\ 0 \ , \ & k = 0  \ ,\\ \sum_{j=0}^{k-1} F \circ \phi^j \ , \ & k \geq 1 \ . \end{cases}
$$

\begin{lemma}\label{lemadd41}
For $n\in \mathbb{N}$ we have up to a $m$ null set that
$$
M_{I_{n}} \subseteq \left\{ x \in X: \ \phi^n(x) = x,  \ \phi^j(x) \neq x, \ 1 \leq j < n, \ \ S_n(F)(x) = 0 \right\} \ .
$$
and we have up to a $m$ null set that 
$$
M_{I_{\infty}} \subseteq \left\{ x \in X: \
\sum_{n \in \mathbb Z}  e^{\beta S_n(F)(x)} \  < \ \infty \right\} \ .
$$
\end{lemma}

\begin{proof}
By definition there is a Borel isomorphism $\psi$ from $m$-almost all of $M_{I_{n}}$ to $\mu$-almost all of a space $S\times X'$ where the measure $\mu=\lambda \times \mu'$ is a product of the counting measure $\lambda$ on $S$ and a measure $\mu'$ on $X'$ and $|S|=n$ for $n\in \mathbb{N}\cup\{ \infty\}$. Defining a relation $\mathcal{R}'$ on $S\times X'$ by setting $(s,x)\sim (t,y)$ if and only if $x=y$ the map $\psi$ satisfies that
\begin{equation*}
\psi\left( \{ y\in M_{I_{n}}\ | \ (y, x) \in \mathcal{R} \} \right) = 
 \{ z\in S\times X'\ | \ (z , \psi(x))\in \mathcal{R}' \} \ 
\end{equation*}
for $m$-almost all $x\in M_{I_{n}}$ and that $m|_{M_{I_{n}}}\circ \psi^{-1}$ is equivalent to $\mu$. If $n\in \mathbb{N}$ then \eqref{eqadd1} implies that $m$-almost all points in $M_{I_{n}}$ has an orbit of size $n$. This implies for any Borel set $B  \subseteq M_{I_{n}}$ that 
$$
m(B)= m(\phi^{n}(B)) = \int_{B} e^{\beta S_{n}(F)(x)} \ \mathrm{d} m (x) \ ,
$$
which proves the statement for $n\in \mathbb{N}$. If $n= \infty$ the set $M_{s}: =\psi^{-1}(\{s\}\times X')$ for a $s\in S$ has positive measure in $M_{I_{\infty}}$, and it satisfies that $\phi^{k}(M_{s})\cap \phi^{l}(M_{s})$ is a null-set for all $k\neq l$. Hence
$$
1 \geq \sum_{n\in \mathbb{Z}}m\left (  \phi^{n}(M_{s}) \right)
=\sum_{n\in \mathbb{Z}} \int_{M_{s}} e^{\beta S_{n}(F)(x)}\ \mathrm{d}m
=\int_{M_{s}}\sum_{n\in \mathbb{Z}}  e^{\beta S_{n}(F)(x)}\ \mathrm{d}m
$$
and since $\bigcup_{s\in S} M_{s}=M_{I_{\infty}}$ this proves the inclusion for $n=\infty$.
\end{proof}

To describe the set $M_{II_{1}}$ we need two general results for $e^{\beta F}$-conformal measures.

\begin{lemma}\label{25-10-19} Let $\beta \in \mathbb R \backslash \{0\}$.
\begin{itemize}
\item[1)] An atomic $e^{\beta F}$-conformal measure is ergodic if and only if it is concentrated on a single $\phi$-orbit.
\item[2)] Let $x \in X$. If $\phi^p(x) = x$ for some $p \in \mathbb N$, there is an $e^{\beta F}$-conformal measure concentrated on $x$'s $\phi$-orbit if and only 
$$
S_p(F)(x) = 0 \ .
$$
\item[3)] Let $x \in X$ and assume that $x$ is not $\phi$-periodic. There is an  $e^{\beta F}$-conformal measure concentrated on $x$'s $\phi$-orbit if and only
\begin{equation}\label{eq73}
\sum_{k\in \mathbb Z} e^{\beta S_k(F)(x)} \  < \ \infty \ .
\end{equation} 
\end{itemize}
\end{lemma}
\begin{proof} Left to the reader.
\end{proof}

\begin{prop}\label{propadd4}
Let $m_{0}$ be an $e^{\beta F}$ conformal measure. Then $m_{0}$ is equivalent to a $\sigma$-finite $\phi$-invariant measure $\nu$ if and only if there is a Borel function $u: X \to \mathbb{R}$ such that $\beta F(x)=u(\phi(x))-u(x)$ for $m_{0}$ almost all $x\in X$. In this case the function $u$ can be chosen as $\ln(k)$ where $\mathrm{d} m_{0} = k \mathrm{d} \nu$.
\end{prop}

\begin{proof}
Assume that $\mathrm{d} m_{0} = k \mathrm{d} \nu$ with $k$ positive. We find that
\begin{equation*}
\begin{split}  & \int_{\phi^{-1}(A)} k \circ \phi \ \mathrm{d} \nu \ = \ \int_{X} (1_A \circ \phi) k \circ \phi \ \mathrm{d} \nu \ = \ \int_X 1_A k \ \mathrm{d}\nu \\
& = \ m_{0}(A) \ = \ \int_{\phi^{-1}(A)} e^{\beta F} \ \mathrm{d} m_{0}  = \int_{\phi^{-1}(A)} e^{\beta F} k \ \mathrm{d} \nu  \ .
 \end{split}
\end{equation*}
and hence  $k\circ \phi(x) = e^{\beta F(x)} k(x)$ for $\nu$-almost every $x$. It is now clear that $u:=\ln(k)$ satisfies the criterion. For the other direction assume that we have a Borel function $u: X \to \mathbb{R}$ such that $\beta F(x)=u(\phi(x))-u(x)$ for $m_{0}$ almost all $x\in X$. Setting $\mathrm{d} \nu = e^{-u}\mathrm{d} m_{0}$ we see that
\begin{equation*}
\int_{X} f \ \mathrm{d} \nu 
= \int_{X} f \ e^{-u} \mathrm{d} m_{0}
= \int_{X} f\circ \phi \ e^{-u\circ \phi} e^{\beta F} \ \mathrm{d} m_{0}
= \int_{X} f\circ \phi \ \mathrm{d} \nu \ ,
\end{equation*}
which proves the Proposition.
\end{proof}

We can now describe the set $M_{II_{1}}$. For this we set
\begin{align*}
M_{I}&:= \left\{ x \in X: \
\sum_{n \in \mathbb Z}  e^{\beta S_n(F)(x)} \  < \ \infty \right\} \cup\\
&\bigcup_{n \in \mathbb{N}} \left\{ x \in X: \ \phi^n(x) = x,  \ \phi^j(x) \neq x, \ 1 \leq j < n, \ \ S_n(F)(x) = 0 \right\}  \ .
\end{align*}

\begin{lemma}\label{lemaddp2}
Up to an $m$-null set then 
$$
M_{II_{1}} \subseteq \left\{x \in X: \ 
\liminf_n \frac{1}{n} \sum_{j=1}^{n} e^{ \beta S_j(F)(x)} \ > \  0  \ \right\} \backslash M_I \
$$

\end{lemma}

\begin{proof}
Define $m_{0}$ to be the restriction of $m$ to $M_{II_{1}}$. By definition $m_{0}$ is equivalent to an invariant probability measure $\nu$, so by Proposition \ref{propadd4} we have a positive Borel function $k \in L^{1}(\nu)$ such that $k\circ \phi(x) = e^{\beta F(x)} k(x)$ for $\nu$-almost every $x$. By the pointwise ergodic theorem we have that
$$
\lim_{n\to \infty} \frac{1}{n} \sum_{i=0}^{n-1} k \circ \phi^{i}(x)
$$
converges for $\nu$-almost all $x$ to a positive number, yet since 
$$
k \circ \phi^{i}(x) = e^{\beta S_{i}(F)(x)} k(x)
$$
this implies that
$$
\frac{1}{n} \sum_{i=0}^{n-1} e^{\beta S_{i}(F)(x)} 
$$
converges to a positive number for $\nu$-almost all $x$. If $m_{0}(M_{II_{1}}\cap M_{I})>0 $ the relation $\mathcal{R}$ restricted to  $M_{II_{1}}\cap M_{I}$ is of type $II_{1}$. It follows from Proposition 3.2 in \cite{FM} that there then also exists an ergodic measure $m'$ on $M_{II_{1}}\cap M_{I}$ such that $m'$ is type $II_{1}$. By a result of K. Schmidt, Theorem 1.2 in \cite{Sc}, there is a $\phi$-invariant Borel set $B$ with $m'(B)=1$ and the property that $m''(B)=0$ for all ergodic $e^{\beta F}$-conformal measures $m''\neq m'$. However this contradicts Lemma \ref{25-10-19}, proving the Lemma.
\end{proof}

To describe the sets $M_{II_{\infty}}$ and $M_{III}$ we need the following result.

\begin{lemma}\label{14-08-19} Let $m$ be an $e^{\beta F}$-conformal measure. There is a (possibly empty) $\phi$-invariant Borel set $N \subseteq X$ with the properties that
\begin{itemize}
\item there is a Borel function $u : X \backslash N \to \mathbb R$ such that $\beta F(x) = u (\phi(x)) - u(x)$ for all $x \in X \backslash N$ and
\item when $A \subseteq N$ is a $\phi$-invariant Borel subset and $u : A \to \mathbb R$ a Borel function such that $\beta F(x) = u (\phi(x)) - u(x)$ for all $x \in A$ then $m(A) = 0$. 
\end{itemize}
\end{lemma}
\begin{proof} Call a $\phi$-invariant Borel set $A \subseteq X$ admissible when $m(A) > 0$ and there is a Borel function $u : A\to \mathbb R$ such that $\beta F(x) = u(\phi(x)) - u(x)$ for all $x \in A$. Assume that there is an admissible Borel set. By Zorns lemma there is then a maximal collection $A_i, i \in I$, of mutually disjoint admissible Borel sets. Note that this set is countable since $\sum_{i \in I} m(A_i)  \leq 1$. Then $N = X \backslash \bigcup_{i \in I} A_i$ will have the stated properties. If there are no admissible Borel sets, set $N = X$.
\end{proof}

\begin{lemma}
The set $M_{III}$ is contained in $N$ up to a $m$ null set.
\end{lemma}

\begin{proof}
Assume that $M_{III}\cap N^{C}$ has positive $m$-measure, and let $m_{0}$ denote the restriction of $m$ to this set. It follows from Lemma \ref{14-08-19} that there is a Borel function $u: X\to \mathbb{R}$ such that $\beta F(x)=u(\phi(x))-u(x)$ for $m_{0}$ almost all $x$. From Proposition \ref{propadd4} we get that $m_{0}$ is equivalent to an invariant measure, contradicting that $m$ restricted to $M_{III}$ is of type $III$. 

\end{proof}

Combining all of the above we now get the following description of the decomposition of the space $X$.

\begin{thm}
For the orbit equivalence relation on the set $X$ and the $e^{\beta F}$-conformal measure $m$ we get the following description of the decomposition of $X$ in \eqref{eqadd1}: 
\begin{itemize}
\item For all $n\in \mathbb{N}$ then 
$$
M_{I_{n}}=\left\{ x \in X: \ \phi^n(x) = x,  \ \phi^j(x) \neq x, \ 1 \leq j < n, \ \ S_n(F)(x) = 0 \right\} \ .
$$
\item $M_{I_{\infty}} = \left\{ x \in X: \
\sum_{n \in \mathbb Z}  e^{\beta S_n(F)(x)} \  < \ \infty \right\} $ .
\item $M_{II_{1}} = \left\{x \in X: \  
\liminf_n \frac{1}{n} \sum_{j=1}^{n} e^{ \beta S_j(F)(x)} \ > \  0  \ \right\} \setminus M_{I}$ .
\item $M_{III}=N$ where $N$ is defined in Lemma \ref{14-08-19}. 
\item $M_{II_{\infty}}=X\setminus (M_{I}\cup M_{II_{1}}\cup M_{III})$ .  
\end{itemize}
\end{thm}

\begin{proof}
By \eqref{eqadd1} it suffices to prove that the five sets $N$, $M_{II_{\infty}}$,
\begin{equation} \label{eqadd2}
\left\{ x \in X: \
\sum_{n \in \mathbb Z}  e^{\beta S_n(F)(x)} \  < \ \infty \right\} \ ,
\end{equation}
$$
\bigcup_{n \in \mathbb{N}}\left\{ x \in X: \ \phi^n(x) = x,  \ \phi^j(x) \neq x, \ 1 \leq j < n, \ \ S_n(F)(x) = 0 \right\}  \ \text{ and }
$$
\begin{equation}\label{eqadd3}
\left\{x \in X: \  
\liminf_n \frac{1}{n} \sum_{j=1}^{n} e^{ \beta S_j(F)(x)} \ > \  0  \ \right\} \setminus M_{I}
\end{equation}
are all disjoint up to $m$ null sets. The last three sets are easily seen to be disjoint from each other. Furthermore it is easy to check that defining a function $k$ on the last two sets by
$$
k(x)= \left(\liminf_n \frac{1}{n} \sum_{j=1}^{n} e^{ \beta S_j(F)(x)} \right)^{-1}
$$ 
we get that $k(\phi(x))=e^{\beta F(x) } k(x)$, and hence by taking logarithms and comparing to Lemma \ref{14-08-19} it follows that these two sets are disjoint from $N$. A similar argument works by considering the function
$$
k(x)= \left (\sum_{n \in \mathbb Z}  e^{\beta S_n(F)(x)} \right)^{-1}
$$
for $x$ in the set \eqref{eqadd2}.

Since $m$ restricted to $M_{II_{\infty}}$ is equivalent with an infinite invariant measure it follows from Proposition \ref{propadd4} that $M_{II_{\infty}}$ is disjoint from $N$. If $m(M_{II_{\infty}}\cap M_{I})>0 $ then Proposition 3.2 in \cite{FM} implies that there exists an ergodic measure $m_{0}$ on $M_{II_{\infty}}\cap M_{I}$ such that $m_{0}$ is type $II_{\infty}$, and using Theorem 1.2 in \cite{Sc} and Lemma \ref{25-10-19} as in the proof of Lemma \ref{lemaddp2} this gives a contradiction. A similar argument would prove that $M_{II_{\infty}}$ was disjoint from the set in \eqref{eqadd3} if we could prove that an ergodic $e^{\beta F}$ conformal probability measure concentrated on \eqref{eqadd3} was of type $II_{1}$. Assume therefore that $m_{0}$ is concentrated on the set \eqref{eqadd3}.  For $j \in \mathbb N$ and $f \in C(X)$, iterated application of \eqref{eqint} shows that
\begin{equation*}\label{30-12-18a}
\begin{split}
& \int_X f \circ \phi^{-j} \ \mathrm{d} m_{0} =  \int_X f e^{\beta S_j(F)(x)} \ \mathrm{d} m_{0} \ .
\end{split}
\end{equation*}
Hence 
\begin{equation*}\label{30-12-18}
\frac{1}{n}\sum_{j=1}^{n} m_{0} \circ \phi^j \ = \ H_n \ \mathrm{d} m_{0} \ ,
\end{equation*}
where
\begin{equation*}\label{16-01-19e}
H_n = \frac{1}{n} \sum_{j=1}^{n} e^{\beta S_j(F)(x)} \ .
\end{equation*}
By compactness of the Borel probability measures in the weak* topology there is a sequence $n_1 < n_2 < n_3 < \cdots$ in $\mathbb N$ and a Borel probability measure $\nu$ such that
\begin{equation*}\label{02-01-19a}
\lim_{i \to \infty} H_{n_i} \mathrm{d} m_{0} \ = \ \lim_{i \to \infty}\frac{1}{n_i}\sum_{j=1}^{n_i} m_{0} \circ \phi^j \ = \  \mathrm{d} \nu
\end{equation*}
in the weak*-topology. Note that $\nu$ is $\phi$-invariant. Set 
$$
g(x) = \liminf_i H_{n_i}(x) \ .
$$ 
It follows from Fatous lemma that
$$
\int_X fg \ \mathrm{d} m_{0} \leq \liminf_i \int_X f H_{n_i} \ \mathrm{d} m_{0} = \int_X f \ \mathrm{d} \nu
$$
for all non-negative $f \in C(X)$. This shows that $g \ \mathrm{d} m_{0}$ is absolutely continuous with respect to $\nu$.  It follows from the assumption that $g(x) > 0$
for $m_{0}$-almost all $x$ and hence that $m_{0}$ is absolutely continuous with respect $\nu$. This implies that $m_{0}$ is equivalent with a finite invariant measure, proving that measures concentrated on \eqref{eqadd3} are of type $II_{1}$. This proves that all five sets are disjoint, which proves the Theorem.
\end{proof}

The above theorem describes the decomposition of an equivalence relation in terms of $\beta F$ on the set $X$. We say that $m$ is of type $I$ when $m(M_{I}) = 1$. In Theorem \ref{15-08-19e} below we will describe the type of ergodic $e^{\beta F}$-conformal measures in terms which will make it easy to describe the corresponding von Neumann algebras. Before doing this we need a small elementary lemma.

\begin{lemma} \label{22-08-19}
An ergodic $e^{\beta F}$-conformal measure $m$ can only be equivalent to one $\sigma$-finite $\phi$-invariant measure up to multiplication by scalars.
\end{lemma}

\begin{proof}
Assume that $\mu_{1}$ and $\mu_{2}$ are equivalent to $m$ and that they both are $\sigma$-finite and $\phi$-invariant. By the Radon-Nikodym theorem there is a $\mu_{2}$ measurable function $f$ such that
\begin{equation*}
\mu_{1}(B)=\int_{B} f \ \mathrm{d}\mu_{2} 
\end{equation*}  
for all Borel sets $B$. Since 
\begin{align*}
& \int_{B} f \ \mathrm{d}\mu_{2}  = \mu_{1}(B)=\mu_{1}(\varphi(B))=\int_{\varphi(B)} f \ \mathrm{d}\mu_{2} \\
&= \int_{X} 1_{B}\circ \varphi^{-1} f \ \mathrm{d}\mu_{2} 
=\int_{B}  f\circ \varphi \ \mathrm{d}\mu_{2} \ ,
\end{align*}  
it follows that $f=f\circ \varphi$ almost everywhere. The ergodicity of $m$ implies that $f$ is constant, completing the proof.
\end{proof}

\begin{thm}\label{15-08-19e} Let $m$ be an ergodic $e^{\beta F}$-conformal measure. 
\begin{itemize}
\item[$\bullet$] \ \ $m$ is of type $I_p$ if and only if $m$ is atomic and concentrated on an $F$-cyclic $\phi$-orbit with minimal period $p$.
\item[$\bullet$] \ \ $m$ is of type $I_{\infty}$ if and only if $m$ is atomic and concentrated on a single infinite $\phi$-orbit.
\item[$\bullet$] \ \ $m$ is of type $II_1$ if and only if $m$ is equivalent to a $\phi$-invariant non-atomic Borel probability measure.
\item[$\bullet$] \ \ $m$ is of type $II_{\infty}$ if and only $m$ is equivalent to an infinite $\sigma$-finite non-atomic $\phi$-invariant measure.
\item[$\bullet$] \ \ $m$ is of type $III$ if and only if $m$ is singular with respect to all $\sigma$-finite $\phi$-invariant Borel measures. 
\end{itemize}
\end{thm}

\begin{proof} By Theorem 1.2 in \cite{Sc}, there is a $\phi$-invariant Borel set $B$ with $m(B)=1$ and the property that $m'(B)=0$ for all ergodic $e^{\beta F}$-conformal measures $m'\neq m$. In combination with 2) and 3) of Lemma \ref{25-10-19} the statements on the type $I_p$ and the type $I_{\infty}$ cases follow from this. Then, by 1) of Lemma \ref{25-10-19}, in the remaining cases $m$ is not atomic, and the last three cases then follows from Lemma \ref{22-08-19} and the definition of these types, c.f. Definition 3.3 in \cite{FM}.
\end{proof}


\begin{remark}\label{23-10-19e} In a recent preprint K. Athanassopoulos describes various results pertaining to questions about conformal measures equivalent to invariant ones, and he describes a sufficient condition for the existence of a conformal measure of type $II_1$ equivalent to a given ergodic invariant measure, cf. Main Result in \cite{A}.
\end{remark} 

\subsection{Factor types and crossed products}

Let $m$ be an $e^{\beta F}$-conformal measure and consider the elements of $L^{\infty}(m)$ as multiplication operators on $L^2(m)$.  Define $\pi : C(X) \to B\left(l^2(\mathbb Z, L^2(m))\right)$ such that
$$
\left(\pi(f)\psi\right)(k) \ = \ f\circ \phi^k \psi(k) 
$$
and $\lambda_k \in B\left(l^2(\mathbb Z, L^2(m))\right)$ such that
$$
\left(\lambda_k \psi\right)(j) = \psi(j-k) \ .
$$
Then $\lambda_k \pi(f)\lambda_{-k} \ = \ \pi( f \circ \phi^{-k})$ and we get a $*$-homomorphism $\pi : C(X) \rtimes_{\phi} \mathbb Z \to  B\left(l^2(\mathbb Z, L^2(m))\right)$ such that
$\pi(fU^k) \ = \ \pi(f)\lambda_k$. Let $u \in l^2(\mathbb Z, L^2(m))$ be the element such that $u(k) = 0$ when $k\neq 0$ and $u(0) = 1$. The triple $(l^2(\mathbb Z, L^2(m)), \pi, u)$ is then (isomorphic to) the GNS-representation of the $\beta$-KMS state $\omega_m$. Since $\pi\left(C(X)\rtimes_{\phi} \mathbb Z\right)''$ is $*$-isomorphic to the von Neumann algebra crossed product $L^{\infty}(m) \rtimes_{\phi} \mathbb Z$, cf. \cite{KR}, we conclude that

\begin{lemma}\label{06-09-19} The von Neumann algebra generated by the GNS-representation of $\omega_m$ is isomorphic to  $L^{\infty}(m) \rtimes_{\phi} \mathbb Z$.
\end{lemma}

Note that $\phi$ acts freely on $L^{\infty}(m)$ when $m$ annihilates all $\phi$-periodic orbits. We can therefore combine Theorem \ref{15-08-19e} above with Proposition 8.6.10 in \cite{KR} to  obtain the following

\begin{thm}\label{15-08-19g} Let $m$ be an $e^{\beta F}$-conformal measure and $\pi_m$ the GNS representation of the $\beta$-KMS state $\omega_m$, cf. Lemma \ref{lemma21}.
\begin{itemize}
\item $
\pi_m\left(C(X) \rtimes_{\phi} \mathbb Z\right)'' \simeq L^{\infty} (m) \rtimes_{\phi} \mathbb Z \ . $
\item Let $t \in \left\{I_{\infty},II_1,II_{\infty}, III\right\}$ and assume that $m$ is ergodic. Then
$\pi_m\left(C(X) \rtimes_{\phi} \mathbb Z\right)''$ is a factor of type $t$ if and only if $m$ is of type $t$.
\end{itemize}
\end{thm}

When $m$ is of type $I_p$ the $\beta$-KMS state $\omega_m$ is not extremal, $\phi$ does not act freely on $L^{\infty}(m)$ and $\pi_m\left(C(X) \rtimes_{\phi} \mathbb Z\right)''$ is not a factor; it's a copy of $M_p(\mathbb C)\otimes L^{\infty}(\mathbb T)$.

\section{$e^{\beta F}$-conformal measures of all types}\label{types}

It is time to show that the theory we have developed is not vacuous by showing that all factor types actually occur; even when we restrict to minimal homeomorphisms. To do so in order, note that Example \ref{31-07-19a} provides us with $e^{\beta F}$-conformal measures of type $I_p$. The easiest and most informative way of realizing type $II_1$-examples is perhaps the following whose proof we leave to the reader.

\begin{prop}\label{03-10-19} Let $H : X \to \mathbb R$ be a continuous function and set $F = H \circ \phi - H$. For all $\beta \in \mathbb R$ the map 
$$
\lambda \mapsto \frac{e^{\beta H}}{\int_X e^{\beta H} \mathrm{d}\lambda }\mathrm{d}\lambda
$$ 
is a homeomorphism from the set of $\phi$-invariant Borel probability measures $\lambda$ onto the set of $e^{\beta F}$-conformal measures.
\end{prop}

In combination with Theorem \ref{08-03-19c} it follows that all  $\beta$-KMS states correspond to $e^{\beta F}$-conformal measures of type $II_1$ when $\alpha^F$ inner and there are no finite $F$-cyclic $\phi$-orbits.

It is considerably harder to realize the remaining three types, but they all occur for well-chosen diffeomorphisms of the circle. This follows from work of Katznelson, \cite{K}, Herman, \cite{He}, and Douady and Yoccoz, \cite{DY}, as we now explain.

\subsection{Diffeomorphisms of the circle and $e^F$-conformal measures of type $I_{\infty}$, $II_{\infty}$ and $III$}

To exhibit $e^{\beta F}$-conformal measures of type $II_{\infty}$ and type $III$ we rely on the following result by Katznelson, which is a combination of Theorem 3.2 and Theorem 3.3 from Part II of \cite{K}.

\begin{thm}\label{Katz}[Katznelson] 
Let $\lambda$ denote Haar measure on the circle $\mathbb{T}$. 
\begin{enumerate}
\item There exists an orientation preserving $C^{\infty}$-diffeomorphism $\phi_{1}$ of $\mathbb{T}$ with irrational rotation number such that there exists an infinite $\sigma$-finite $\phi_{1}$-invariant measure $m$ equivalent to $\lambda$.
\item There exists an orientation preserving $C^{\infty}$-diffeomorphism $\phi_{2}$ of $\mathbb{T}$ with irrational rotation number such that there does not exist any $\sigma$-finite $\phi_{2}$-invariant measure equivalent to $\lambda$.
\end{enumerate}
\end{thm}

\begin{lemma}\label{02-10-19}
Let $\phi$ be an orientation preserving $C^{2}$ diffeomorphism of $\mathbb{T}$ with irrational rotation number, and let $\lambda$ denote the Haar measure on $\mathbb{T}$. Then $\phi$ is minimal and $\lambda$ is $\phi$-ergodic. Furthermore, 
\begin{equation} \label{etype2}
\lambda(\phi(B))=\int_{B} D(\phi) \ \mathrm{d} \lambda 
\end{equation}
for all Borel sets $B\subseteq \mathbb{T}$, where $D(\phi)$ is the differential of $\phi$. 
\end{lemma}
\begin{proof} $\phi$ is topological conjugate to an irrational rotation by a result of Denjoy, \cite{De}, and hence minimal. That $\lambda $ is $\phi$-ergodic follows from Theorem 1.4 in Chapter VII in \cite{He}. \eqref{etype2} is wellknown and easy to prove.
\end{proof}

The existence of $e^{F}$-conformal measures of type $II_{\infty}$ and type $III$ follows now by combining Lemma \ref{02-10-19} with Katznelson's Theorem \ref{Katz} as follows: For $\phi_{i}$ set $F:= \log D(\phi_{i})$. Then $\lambda$ is $e^{F}$-conformal and ergodic for $\phi_{i}$ by Lemma \ref{02-10-19}. Theorem \ref{15-08-19e} implies that $\lambda$ is of type $II_{\infty}$ for $i=1$ and of type $III$ for $i=2$. By appealing to results we shall prove below, we can offer the following additional information about the examples of Katznelson:
\begin{itemize}
\item For both $i = 1$ and $i=2$ there is a unique $D(\phi_i)^{\beta}$-conformal measure for all $\beta \in \mathbb R$. 
\end{itemize} The statement concerning existence follows from Theorem \ref{18-10-19b} or Theorem \ref{Hurra} while the uniqueness is a consequence of the uniqueness statement in the following theorem which is a generalisation of Th\'eor\`eme 1 in \cite{DY}.

\begin{thm}\label{02-10-19c}
Assume $\phi$ is a homeomorphism of the circle topological conjugate to an irrational rotation $R_{\alpha}$, i.e. there is a homeomorphism $h$ with $\phi=h\circ R_{\alpha} \circ h^{-1}$. There exists $\beta$-KMS states for $\alpha^{F}$ for some $\beta \neq 0$ if and only if 
$$
\int_{\mathbb T}
 F\circ h \ d\lambda =0 \ ,
$$ 
in which case they exist for all $\beta \in \mathbb{R}$. If $F:\mathbb{T}\to \mathbb{R}$ has bounded variation the $\beta$- KMS state is unique for each $\beta \in \mathbb R$.
\end{thm}
\begin{proof}  The statement concerning existence follows from Theorem \ref{18-10-19b} below because $\lambda \circ h^{-1}$ is the only $\phi$-invariant Borel probability measure. The uniqueness follows from an obvious adaptation of the proof of Lemme 1 in \cite{DY}. 
\end{proof}

Theorem \ref{02-10-19c} applies to the examples by Katznelsom because his diffeomorphisms are $C^{\infty}$. In particular, they are conjugate to a rotation by the result of Denjoy, \cite{De}. We have no idea what the factor types of the $D(\phi_i)^{\beta}$-conformal measures are in Katznelsons examples when $\beta \notin \{0,1\}$.


The paper by Douady and Yoccoz contains examples of irrational rotations of the circle for which there is a potential $F : \mathbb T \to \mathbb R$ such that there are two distinct $e^{F}$-conformal measures, both of type $I_{\infty}$. The construction in \cite{DY} can be modified to provide $n$ distinct $e^{F}$-conformal measures, but their construction only works for rotations on the circle for which the rotation number satisfies a certain diophantic equation, and will not work for e.g. the Liouville numbers.  See Section 7.2 in \cite{DY}. It will be shown in Appendix \ref{AppA} below that there can be arbitrarily many ergodic $e^{\beta F}$-conformal measures of type $I_{\infty}$ for any non-periodic homeomorphisms.

\section{Construction and existence of $e^{\beta F}$-conformal measures}


There exists a method to construct conformal measures, developed by Patterson, \cite{P}, Sullivan, \cite{S} and Denker and Urbanski, \cite{DU}, and it is mainly aimed at non-invertible dynamical systems. We shall here develop an alternate method inspired by Hopf's ratio ergodic theorem for homeomorphisms, cf. e.g. \cite{Ho}.

\begin{lemma}\label{12-09-19a} Assume that there is a point $x\in X$ and two increasing sequences of natural numbers, $\{n_{i}\}_{i=1}^{\infty}$ and $\{m_{i}\}_{i=1}^{\infty}$, such that
\begin{equation}\label{12-09-19l}
\lim_{i\to \infty} \frac{e^{\beta S_{-n_{i}}(F)(x)}+e^{\beta S_{m_{i}}(F)(x)}}{\sum_{k=-n_{i}}^{m_{i}} e^{\beta S_{k}(F)(x)}} \ = \ 0 \ .
\end{equation}
Define
\begin{equation*}
L_{i}(f)\ = \ \frac{\sum_{k=-n_{i}}^{m_{i}} f(\phi^{k}(x))e^{\beta S_{k}(F)(x)} }{\sum_{k=-n_{i}}^{m_{i}} e^{\beta S_{k}(F)(x)}} \ .
\end{equation*}
There is an $e^{\beta F}$-conformal measure $m_x$ and an increasing sequence $\{i_{j}\}_{j=1}^{\infty}$ in $\mathbb N$ such that
\begin{equation*}
\lim_{j \to \infty} L_{i_j}(f) \ = \ \int_X f \ \mathrm{d}m_x
\end{equation*}
for all $f \in C(X)$. 
\end{lemma}

\begin{proof}
Since $C(X)$ is separable and $|L_{i}(f)|\leq \lVert f \rVert_{\infty}$, a standard diagonal sequence argument gives an increasing sequence $\{i_{j}\}_{j=1}^{\infty}$ such that $\lim_{j\to \infty} L_{i_{j}}(f)$ exists for all $f\in C(X)$. It follows from the Riesz representation theorem that there is a Borel probability measure $m_{x}$ on $X$ such that
\begin{equation*}
\int_X f \ \mathrm{d} m_x \ = \ \lim_{j \to \infty} L_{i_j}(f) 
\end{equation*}
for all $f \in C(X)$. Since $S_{k}(F)(x)+F(\phi^{k}(x))=S_{k+1}(F)(x)$ for all $k\in \mathbb{Z}$ we get for all $j \in \mathbb{N}$ that
\begin{align*}
&L_{i_{j}}(f)-L_{i_{j}}(e^{\beta F} f\circ \phi ) \\
&  \\
&= \ \frac{\sum_{k=-n_{i_{j}}}^{m_{i_{j}}} f(\phi^{k}(x))e^{\beta S_{k}(F)(x)} }{\sum_{k=-n_{i_{j}}}^{m_{i_{j}}} e^{\beta S_{k}(F)(x)}}
-\frac{\sum_{k=-n_{i_{j}}}^{m_{i_{j}}} f(\phi^{k+1}(x)) e^{\beta F(\phi^{k}(x))}e^{\beta S_{k}(F)(x)} }{\sum_{k=-n_{i_{j}}}^{m_{i_{j}}} e^{\beta S_{k}(F)(x)}} 
\\ 
&\\
&= \ \frac{f(\phi^{-n_{i_{j}}}(x)) e^{\beta S_{-n_{i_{j}}}(F)(x)}  - f(\phi^{m_{i_{j}}+1}(x)) e^{\beta S_{m_{i_{j}}+1}(F)(x)}}{\sum_{k=-n_{i_{j}}}^{m_{i_{j}}} e^{\beta S_{k}(F)(x)}} \\
&=\frac{f(\phi^{-n_{i_{j}}}(x)) e^{\beta S_{-n_{i_{j}}}(F)(x)}  - f(\phi^{m_{i_{j}}+1}(x)) e^{\beta F(\phi^{m_{i_{j}}}(x))}e^{\beta S_{m_{i_{j}}}(F)(x)}}{\sum_{k=-n_{i_{j}}}^{m_{i_{j}}} e^{\beta S_{k}(F)(x)}} \ .
\end{align*}
Since $\lvert f(\phi^{m_{i_{j}}+1}(x))\rvert\leq \lVert f \rVert_{\infty}$ and $ \lvert e^{F(\phi^{m_{i_{j}}}(x))}\rvert \leq \lVert e^{\beta F}\rVert_{\infty}$ for all $j$ it follows from our assumption \eqref{12-09-19l} that
\begin{equation*}
\lvert L_{i_{j}}(f)-L_{i_{j}}(e^{\beta F}f\circ \phi)  \rvert \to 0 \text{ for } j\to \infty \ , 
\end{equation*}
proving that $m_{x}$ is $e^{\beta F}$ conformal.
\end{proof}

\begin{thm}\label{thm31} 
Let $X$ be a compact metric space and let $\phi : X \to X$ be a homeomorphism. Fix a $\beta\in \mathbb{R}$ and a continuous function $F:X\to \mathbb{R}$. The following are equivalent:
\begin{enumerate}
\item There exists an $e^{\beta F}$-conformal measure for $\phi$.
\item There exists a point $x\in X$ such that 
\begin{equation}\label{eq31}
\limsup_k \frac{1}{k} \sum_{i=0}^{k-1} \beta F(\phi^{i}(x)) \ \leq \ 0 
\end{equation}
and 
\begin{equation}\label{eq32}
\limsup_k \frac{1}{k} \sum_{i=1}^{k} - \beta F(\phi^{-i}(x)) \ \leq \ 0 \ .
\end{equation}
\item There exists a point $x\in X$ such that 
\begin{equation}\label{eq33}
\liminf_k \frac{1}{k} \sum_{i=0}^{k-1} \beta F(\phi^{i}(x)) \ \leq \ 0 
\end{equation}
and 
\begin{equation}\label{eq34}
\liminf_k \frac{1}{k} \sum_{i=1}^{k} - \beta F(\phi^{-i}(x)) \ \leq \ 0 \ .
\end{equation}
\end{enumerate}
\end{thm}

\begin{proof}
$1) \Rightarrow 2)$: Assume that there exists an $e^{\beta F}$-conformal measure. By Krein-Milmans Theorem there is also an extremal $e^{\beta F}$-conformal measure $m$, which is then $\phi$-ergodic. Consider the Borel functions $G$ and $H$ on $X$ given by
\begin{equation*}
H(x):=\limsup_{k} \frac{1}{k} \sum_{i=1}^{k} - \beta F \circ \phi^{-i}(x)  
\end{equation*}
and 
\begin{equation*}
G(x):=\limsup_{k} \frac{1}{k} \sum_{i=0}^{k-1}\beta F \circ \phi^{i}(x)  \ ,
\end{equation*}
respectively. They are both $\phi$-invariant and must therefore be constant $m$-almost everywhere since $m$ is $\phi$-ergodic, i.e. there exist $c_{1},c_{2} \in \mathbb{R}$ such that $G(x)=c_{1}$ and $H(x)=c_{2}$ for $m$-almost all $x$. Assume for a contradiction that $c_{1}>0$ and set
\begin{equation*}
M_{n} = \left\{ x \in X: \  \frac{1}{n} \sum_{i=0}^{n-1}\beta F \circ \phi^{i}(x) \geq \frac{c_{1}}{2} \right\} \ .
\end{equation*}
For all $N \in \mathbb N$,
\begin{equation}\label{eqsum}
X = \bigcup_{n \geq N} M_n \ ,
\end{equation}
up to an $m$-null set. Using \eqref{eqint} for $f=1$ gives $1=\int_{X}e^{\beta F}\mathrm{d} m$. In fact, iterated applications of the formula \eqref{eqint} show that
\begin{equation*}
1 = \int_X \exp \left( \sum_{i=0}^{n-1} \beta F \circ \phi^i \right) \ \mathrm{d} m \ \text{ for all } n\geq 1 \ ,
\end{equation*}
and hence
\begin{align*}
&1 \ \geq \ \int_{M_n} \exp \left( \sum_{i=0}^{n-1} \beta F \circ \phi^i \right) \ \mathrm{d} m \\
&  \geq \  \int_{M_n} \exp \left( n \frac{c_{1}}{2} \right)  \ \mathrm{d} m  \ =  \ m(M_n) \exp\left(\frac{nc_{1}}{2}\right) \ .
\end{align*} 
It follows that $m(M_n) \leq \exp\left(-n \frac{c_{1}}{2}\right)$ for all $n$  and in combination with \eqref{eqsum} that
\begin{equation*}
1 = m(X) \leq \sum_{n=N}^{\infty} \exp \left(-n \frac{c_{1}}{2}\right)
\end{equation*}
for all $N$, contradicting that $\lim_{N \to \infty} \sum_{n=N}^{\infty} \exp\left(-n \frac{c_{1}}{2}\right) = 0$. Hence $c_{1}\leq 0$. Since \eqref{eqint} can also be used to prove that:
\begin{equation*}
1= \int_X \exp \left( \sum_{i=1}^{n}- \beta F \circ \phi^{-i} \right) \ \mathrm{d} m \ \text{ for all } n\geq 1 \ ,
\end{equation*}
a similar argument gives that $c_{2}\leq 0$. In conclusion $m$-almost all points satisfy \eqref{eq31} and \eqref{eq32},  completing the proof of $1)\Rightarrow 2)$.

 $2) \Rightarrow 3)$ is obvious and we will argue that $3)$ implies $1)$ to complete the proof. Assume therefore that there exists a point $x\in X$ satisfying \eqref{eq33} and \eqref{eq34}. We can assume that
\begin{equation}\label{11-11-19}
\sum_{k=-\infty}^{\infty} e^{\beta S_{k}(F)(x)} \ = \ \infty  \end{equation}
since otherwise we get the desired $e^{\beta F}$-conformal measure from 3) of Lemma \ref{25-10-19}.  As a crucial step we prove

\emph{Claim 1: There exists an increasing sequence $\{m_{i}\}_{i=1}^{\infty}$ in $\mathbb N$ such that for any increasing sequence $\{n_{i}\}_{i=1}^{\infty}$ in $\mathbb N$ we have
\begin{equation*}
\lim_{i\to \infty} \ \frac{e^{\beta S_{m_{i}}(F)(x)}}{\sum_{k=-n_{i}}^{m_{i}} e^{\beta S_{k}(F)(x)}} \ =\ 0 \ .
\end{equation*}}

To prove Claim 1 assume first that there exists an increasing infinite sequence $\{m_{i}\}_{i=1}^{\infty}$ with the property that
\begin{equation}\label{ee1}
\sum_{j=0}^{m_{i}-1} \beta F ( \phi^{j}(x)) \leq 0 \text{ for all } i\in \mathbb{N} \ .
\end{equation}
Then for any increasing sequence $\{n_{i}\}_{i=1}^{\infty}$ we get
\begin{equation*}
0 \ \leq \ \frac{e^{\beta S_{m_{i}}(F)(x)}}{\sum_{k=-n_{i}}^{m_{i}} e^{\beta S_{k}(F)(x)}} \ \leq \ \frac{1}{\sum_{k=-n_{i}}^{m_{i}} e^{\beta S_{k}(F)(x)}} \ .
\end{equation*}
It follows from \eqref{11-11-19} that the righthand side converges to $0$, and we can therefore assume that there is no infinite sequence satisfying \eqref{ee1}. This implies that we can choose a $m_{1}\in \mathbb{N}$ so that 
\begin{equation*}
\sum_{j=0}^{l-1} \beta F ( \phi^{j}(x)) > 0 
\end{equation*}
for all $l\geq m_{1}$. Combined with \eqref{eq33} this implies that
\begin{equation*}
\inf_{l\geq m_{1}} \frac{1}{l} \sum_{i=0}^{l-1} \beta F ( \phi^{i}(x)) =0 \ ,
\end{equation*}
and hence we can choose an increasing sequence $\{m_{i}\}_{i=1}^{\infty}$ such that for every $i\in \mathbb{N}$ and every $l\in [m_{1}, m_{i}[$ we have
\begin{equation} \label{e123}
\frac{1}{l} \beta S_{l}(F)(x) =
\frac{1}{l} \sum_{j=0}^{l-1} \beta F ( \phi^{j}(x)) >
\frac{1}{m_{i}} \sum_{j=0}^{m_{i}-1} \beta F ( \phi^{j}(x))=\frac{1}{m_{i}} \beta S_{m_{i}}(F)(x) \ .
\end{equation}
Define the decreasing sequence $\{c_{i}\}_{i=1}^{\infty}$ of positive reals by setting 
$$
c_{i}:=\exp\Big(\frac{1}{m_{i}} \beta S_{m_{i}}(F)(x)\Big)
$$ 
for each $i$. It follows from the choice of $\{m_{i}\}_{i=1}^{\infty}$ that $\lim_{i\to \infty} c_{i}=1$. Take an arbitrary increasing sequence $\{n_{i}\}_{i=1}^{\infty}$. If $\{c_{i}^{m_{i}}\}_{i=1}^{\infty}$ does not converge to infinity then it is bounded for some subsequence $i_{j}$ and 
\begin{equation*}
\frac{e^{\beta S_{m_{i_{j}}}(F)(x)}}{\sum_{k=-n_{j}}^{m_{i_{j}}} e^{\beta S_{k}(F)(x)}}
=\frac{c_{i_{j}}^{m_{i_{j}}}}{\sum_{k=-n_{j}}^{m_{i_{j}}} e^{\beta S_{k}(F)(x)}}
\end{equation*}
will converge to zero for $j\to \infty$ by \eqref{11-11-19}, and we have established the claim. Assume therefore that $c_{i}^{m_{i}}\to \infty$ for $i\to \infty$ and let $\{n_{i}\}_{i=1}^{\infty}$ be any increasing sequence in $\mathbb N$. Using \eqref{e123} we get that
\begin{align*}
& 0 \ \leq \ \frac{e^{\beta S_{m_{i}}(F)(x)}}{\sum_{k=-n_{i}}^{m_{i}} e^{\beta S_{k}(F)(x)}} \
\leq \ \frac{c_{i}^{m_{i}}}{\sum_{k=m_{1}}^{m_{i}} e^{\beta S_{k}(F)(x)}} \\
&\leq \ \frac{c_{i}^{m_{i}}}{\sum_{k=m_{1}}^{m_{i}} c_{i}^{k}} \ 
= \ \left(\sum_{k=0}^{m_{i} -m_{1}} (\frac{1}{c_{i}})^{k}\right)^{-1} \ .
\end{align*}
Since $\lim_{i\to \infty}c_{i}=1$ and $\lim_{i\to \infty}c_{i}^{-m_{i}}=0$ we deduce from the formula
\begin{equation*}
1=\bigg(1-\frac{1}{c_{i}}\bigg)\sum_{k=0}^{m_{i} -m_{1}} \bigg(\frac{1}{c_{i}}\bigg)^{k} + \bigg(\frac{1}{c_{i}}\bigg)^{m_{i}} c_{i}^{m_{1}-1} \ ,
\end{equation*}
that $\sum_{k=0}^{m_{i} -m_{1}} (\frac{1}{c_{i}})^{k}\to \infty$ for $i\to \infty$. This finishes the proof of Claim 1.

\smallskip

The condition in \eqref{eq34} is the same as the one in \eqref{eq33} if one replaces $F$ by $-F\circ \phi^{-1}$ and $\phi$ by $\phi^{-1}$ and the sum \eqref{11-11-19} is unchanged by this replacement. In this way Claim $1$ implies 

\smallskip

\emph{Claim 2: There exists an increasing sequence $\{n_{i}\}_{i=1}^{\infty}$ in $\mathbb N$ such that for any increasing sequence $\{m_{i}\}_{i=1}^{\infty}$ in $\mathbb N$ we have
\begin{equation*}
\frac{e^{\beta S_{-n_{i}}(F)(x)}}{\sum_{k=-n_{i}}^{m_{i}} e^{\beta S_{k}(F)(x)}} \to 0 \text{ for } i \to \infty \  .
\end{equation*}}

\smallskip 
It follows from the two claims that there are increasing sequences $\{n_{i}\}_{i=1}^{\infty}$ and $\{m_{i}\}_{i=1}^{\infty}$ in $\mathbb N$ such that
\begin{equation*}
\frac{e^{\beta S_{-n_{i}}(F)(x)}+e^{\beta S_{m_{i}}(F)(x)}}{\sum_{k=-n_{i}}^{m_{i}} e^{\beta S_{k}(F)(x)}} \to 0 \text{ for } i \to \infty \ ,
\end{equation*}
and the existence of an $e^{\beta F}$-conformal measure follows from Lemma \ref{12-09-19a}.  Hence $3) \Rightarrow 1)$.

\end{proof}

\begin{remark}\label{Schmidt}
 Klaus Schmidt mention in \cite{Sc} that it is an open question to decide when there are $e^g$-conformal measures for a given automorphism of a standard Borel space and a given real-valued Borel function $g$. As far as we can tell this problem is still open, but Theorem \ref{thm31} offers an answer when the dynamical system is a homeomorphism and $g$ is continuous.
\end{remark}

\begin{cor}\label{cor34} Let $m$ be an ergodic $e^{\beta F}$-conformal measure. Then \eqref{eq31} and \eqref{eq32} hold for $m$-almost all $x \in X$. 
\end{cor}
\begin{proof} This was established in the proof of Theorem \ref{thm31}. 
\end{proof}

\begin{cor}\label{cor36}  Let $I$ be the \emph{KMS spectrum} for $\alpha^{F}$, i.e. the set of real numbers $\beta$ such that there exists an $e^{\beta F}$-conformal measure for $\phi$. Then $I$ is one of following intervals:
\begin{itemize}
\item $I = \{0\}$,
\item $I = ]-\infty, \  0]$,
\item $I = [0,\infty[$, or
\item $I = \mathbb R$.
\end{itemize}
\end{cor}
\begin{proof} This follows from the observation that the two conditions \eqref{eq31} and \eqref{eq32} hold for $\beta = 0$, and if they hold for some non-zero $\beta_{0}$ they hold for all real numbers $\beta$ with the same sign as $\beta_{0}$.
\end{proof}

\begin{cor}\label{cor37} Assume that there is an ergodic $\phi$-invariant probability measure $\nu$ such that $\int_X F \ \mathrm{d}\nu = 0$. It follows that there is an $e^{\beta F}$-conformal measure for all $\beta \in \mathbb R$.
\end{cor}
\begin{proof} 
Since $\nu$ is a $\phi$-ergodic and $\phi$-invariant measure, it is also $\phi^{-1}$-ergodic and $\phi^{-1}$-invariant. By Birkhoff's pointwise ergodic theorem,
\begin{equation*}
\lim_{n\to \infty} \frac{1}{n} \sum_{i=0}^{n-1} \beta F \circ \phi^i(x) = \int_X \beta F \ \mathrm{d}\nu = 0 
\end{equation*}
and 
\begin{equation*}
\lim_{n\to \infty} \frac{1}{n} \sum_{i=1}^{n} - \beta F \circ \phi^{-i}(x)=\int_X -\beta F \ \mathrm{d}\nu = 0 \ 
\end{equation*}
 for $\nu$-almost all $x \in X$.
Hence Theorem \ref{thm31} applies.
\end{proof}

 In general the condition in Corollary \ref{cor37}, that $\int_X F \ \mathrm{d}\nu = 0$ for an ergodic $\phi$-invariant measure $\nu$, is only a sufficient condition for the existence of an $e^{\beta F}$-conformal measure for all $\beta \in \mathbb R$, but as the next proposition shows it is actually not that far away from also being necessary.

\begin{prop}\label{18-10-19a} Assume that there is an $e^{\beta F}$-conformal measure for some $\beta \neq 0$. It follows that there are ergodic $\phi$-invariant Borel probability measures $\nu_+$ and $\nu_-$ such that $\int_X F  \ \mathrm{d}\nu_+ \geq 0$ and $\int_X F \ \mathrm{d}\nu_- \leq 0$. In particular, a convex combination of $\nu_+$ and $\nu_-$ is a $\phi$-invariant Borel probability measure $\nu$ such that $\int_X F  \ \mathrm{d}\nu =0$.
\end{prop}
\begin{proof} This follows from the more general Theorem 6.8 in \cite{Th}, but we give here a direct proof based on Theorem \ref{thm31}. Let $M(X)$ denote the set of $\phi$-invariant Borel probability measures on $X$. Let $x  \in X$ be a point for which \eqref{eq31} and \eqref{eq32} both hold. Let $\nu_1$ be a condensation point for the weak* topology of the sequence 
$$
\frac{1}{n} \sum_{i=0}^{n-1} \delta_{\phi^i(x)} \ , \ n = 1,2,3,4, \cdots \ .
$$
Then $\nu_1 \in M(X)$ and it follows from \eqref{eq31} that $\beta \int_X F \ \mathrm{d}\nu_1 \leq 0$. In the same way \eqref{eq32} gives rise to a $\nu_2 \in M(X)$ such that $-\beta \int_X F \ \mathrm{d}\nu_2 \leq 0$. Set
$$
\delta_+ = \sup \left\{ \int_X F \ \mathrm{d} \nu : \ \nu \in M(X) \right\} \ .
$$
It follows from the preceding that $\delta_+ \geq 0$. The set
$$
\left\{\nu \in M(X): \ \int_X F \ \mathrm{d} \nu \ = \ \delta_+ \right\}
$$
is a closed non-empty face in $M(X)$ and it contains therefore an extreme point $\nu_+$ of $M(X)$. Then $\nu_+$ is $\phi$-ergodic and $\int_X F \   \mathrm{d}\nu_+ \geq 0$. In the same way we get also a $\phi$-ergodic $\nu_- \in  M(X)$ such that $\int_X F \ \mathrm{d} \nu_- \leq 0$.

\end{proof}

\begin{thm}\label{18-10-19b} Assume that $(X,\phi)$ is uniquely ergodic and let $\mu$ be the $\phi$-invariant Borel probability measure. The following conditions are equivalent:
\begin{itemize}
\item[i)] There is an  $e^{\beta F}$-conformal measure for some $ \beta \neq 0$. 
\item[ii)] There is an  $e^{\beta F}$-conformal measure for all $ \beta \in \mathbb R$. 
\item[iii)] $\int_X F \ \mathrm{d}\mu \ = \ 0 $.
\end{itemize}
\end{thm}
\begin{proof} i) $\Rightarrow$ iii) follows from Proposition \ref{18-10-19a} and iii) $\Rightarrow$ ii) from Corollary \ref{cor37}.
\end{proof}

\begin{remark}\label{DYanswer}
Following equation $(7.1)$ in \cite{DY}, Douady and Yoccoz raise the question about existence and uniqueness of $e^{\beta \psi}$-conformal measures for an irrational rotation $R_{\alpha}$ and for a potential $\psi$ with $\int \psi \ \mathrm{d}\lambda=0$. Theorem \ref{18-10-19b} answers the existence question since it implies that there are $e^{\beta \psi}$-conformal measures for all $\beta$ in this setup. It follows also, from Theorem \ref{02-10-19c}, that the measure is unique for each $\beta$ when $\psi$ has bounded variation. Douady and Yoccoz construct examples to show that uniqueness fails for certain irrational rotations and certain $\psi$. By Theorem \ref{thm73} there are counterexamples to uniqueness for all irrational rotations.

\end{remark}

\begin{remark}\label{24-10-19} It follows from Theorem \ref{18-10-19b} and Theorem \ref{thm41} that when $(X,\phi)$ is uniquely ergodic the flow $\alpha^F$ must be approximately inner if there is a $\beta$-KMS state for $\beta \neq 0$. This is not the case if $(X,\phi)$ is not uniquely ergodic. To show this by example, consider a homeomorphism $\phi$ with exactly two ergodic $\phi$-invariant Borel probability measures $\omega_1$ and $\omega_2$. There is then a potential $F$ such that $\int_X F \ \mathrm{d} \omega_1  = 0$ and  $\int_X F \ \mathrm{d} \omega_2  \neq  0$. By Corollary \ref{cor37} there are $e^{\beta F}$-conformal measures for all $\beta \in \mathbb R$ and hence also $\beta$-KMS states for $\alpha^F$ for all $\beta \in \mathbb R$, but $\alpha^F$ is not approximately inner 
by Theorem \ref{thm41}.
\end{remark}

For any of the intervals $I$ in Corollary \ref{cor36} there exist examples of systems $(X, \phi)$ and potentials $F:X\to \mathbb{R}$ with KMS spectrum $I$. If $F=1$ then $I=\{0\}$ by Corollary \ref{cor35}, and if $F=0$ then $I=\mathbb{R}$ by Corollary \ref{cor37}. To realize the case $I = [0,\infty)$ in Corollary \ref{cor36} in a very simple example, let $X =[0,1]$ and $\phi(x) = x^2$. Take $F$ to be any real valued continuous function on $[0,1]$ which is negative in a neighborhood of $0$ and positive in a neighborhood of $1$. The two conditions \eqref{eq31} and \eqref{eq32} are met iff $\beta \geq 0$, i.e. $I =[0,\infty)$ in this case. The case $I =]-\infty,0]$ arises by exchanging the last $F$ by $-F$. However, if one restricts attention to minimal homeomorphisms the KMS spectrum is either $\{0\}$ or $\mathbb{R}$, as shown in the next theorem.

\begin{thm}\label{Hurra}
Let $X$ be a compact metric space and let $\phi:X\to X$ be a minimal homeomorphism. Assume that $F:X\to \mathbb{R}$ is continuous. The KMS spectrum of $\alpha^{F}$ is $I=\{0\}$ or $I=\mathbb{R}$.
\end{thm}
\begin{proof}
In light of Lemma \ref{lemma22} and Corollary \ref{cor36} it suffices to show that the existence of an $e^{F}$-conformal measure for $\phi$ implies that there exists an $e^{-F}$-conformal measure for $\phi$. Assume therefore that we have an $e^{F}$-conformal measure. By Theorem \ref{thm31} there exists a point $x\in X$ such that
\begin{equation} \label{ei41}
\limsup_{k} \frac{1}{k}\sum_{i=0}^{k-1} F(\phi^{i}(x))\leq 0
\end{equation}
and 
\begin{equation}\label{ei42}
\limsup_{k} \frac{1}{k}\sum_{i=1}^{k} -F(\phi^{-i}(x))\leq 0 \ .
\end{equation}
We construct now a sequence $\left\{U_i\right\}_{i=0}^{\infty}$ of non-empty open sets in $X$ and two increasing sequences $\left\{N_i\right\}_{i=1}^{\infty}$, $\left\{M_i\right\}_{i=1}^{\infty}$ of natural numbers such that
\begin{equation} \label{ei51}
U_{i-1} \supseteq \overline{U_{i}} \supseteq U_{i} 
\end{equation}
and for all $y\in U_{i}$ we have
\begin{equation} \label{ei52}
\frac{1}{N_{i}}\sum_{j=1}^{N_{i}} F(\phi^{-j}(y)) \ < \ \frac{1}{i} \quad\text{ and }\quad \frac{1}{M_{i}} \sum_{j=0}^{M_{i}-1} -F(\phi^{j}(y)) \ < \ \frac{1}{i}
\end{equation}
when $i \geq 1$. To do this by induction set $U_0 = X$. Equation \eqref{ei41} implies that we can find a number $N_{1}$ such that
\begin{equation*}
\frac{1}{N_{1}}\sum_{i=1}^{N_{1}} F(\phi^{-i}(\phi^{N_{1}}(x)))=\frac{1}{N_{1}}\sum_{i=0}^{N_{1}-1} F(\phi^{i}(x)) \ < \ 1 \ .
\end{equation*}
By continuity we can find an open ball $V_{1}$ containing $\phi^{N_{1}}(x)$ such that for all $y\in V_{1}$ we have
\begin{equation*}
\frac{1}{N_{1}}\sum_{i=1}^{N_{1}} F(\phi^{-i}(y)) \ < \ 1 \ .
\end{equation*} 
Using \eqref{ei42} and minimality of $\phi$ we can pick $M_{1}\in \mathbb{N}$ with $\phi^{-M_{1}}(x)\in V_{1}$ such that
\begin{equation*}
\frac{1}{M_{1}}\sum_{i=0}^{M_{1}-1} -F(\phi^{i}(\phi^{-M_{1}}(x))) \ = \
 \frac{1}{M_{1}}\sum_{i=1}^{M_{1}} -F(\phi^{-i}(x))\ < \ 1 \  .
\end{equation*}
By continuity there is an open ball $U_{1} \subseteq V_1$ such that for all $y\in U_{1}$ we have the estimate
\begin{equation*}
\frac{1}{M_{1}}\sum_{i=0}^{M_{1}-1} -F(\phi^{i}(y)) \ < \ 1 \ .
\end{equation*}
This starts the induction. For the induction step assume that for some $n\in \mathbb{N}$ we have two finite increasing sequences $\{N_{i}\}_{i=1}^{n}$ and $\{M_{i}\}_{i=1}^{n}$ of natural numbers and $n+1$ non-empty open sets $\{U_{i}\}_{i=0}^{n}$ such that for each $i\in \{1, \dots , n\}$ and each $y \in U_i$ both \eqref{ei51} and \eqref{ei52} hold. Using \eqref{ei41} and minimality of $\phi$ we can pick a $N_{n+1} >N_{n}$ such that $\phi^{N_{n+1}}(x)\in U_{n}$ and
\begin{equation*}
\frac{1}{N_{n+1}} \sum_{i=1}^{N_{n+1}} F(\phi^{-i}(\phi^{N_{n+1}}(x)))   =
\frac{1}{N_{n+1}}\sum_{i=0}^{N_{n+1}-1} F(\phi^{i}(x)) \ <\ \frac{1}{n+1} \ .
\end{equation*}
By continuity there is a non-empty open set $V_{n+1}\subseteq U_{n}$ such that for all $y\in V_{n+1}$, 
\begin{equation*}
\frac{1}{N_{n+1}} \sum_{i=1}^{N_{n+1}} F(\phi^{-i}(y)) \  <\ \frac{1}{n+1} \ .
\end{equation*}
Now use \eqref{ei42} to pick $M_{n+1}>M_{n}$ such that $\phi^{-M_{n+1}}(x)\in V_{n+1}$ and
\begin{equation*}
\frac{- 1}{M_{n+1}} \sum_{i=0}^{M_{n+1}-1} F(\phi^{i}(\phi^{-M_{n+1}}(x)))=
\frac{-1}{M_{n+1}}\sum_{i=1}^{M_{n+1}} F(\phi^{-i}(x)) \ <\ \frac{1}{n+1} \ .
\end{equation*}
Choose $U_{n+1}$ open and non-empty with $\overline{U_{n+1}} \subseteq V_{n+1}$ such that for all $y\in U_{n+1}$ 
\begin{equation*}
\frac{1}{M_{n+1}} \sum_{i=0}^{M_{n+1}-1} -F(\phi^{i}(y)) \ <\ \frac{1}{n+1}\ .
\end{equation*}
This completes the induction step and hence the construction of the sequences $\left\{U_i\right\}_{i=0}^{\infty}$, $\left\{N_i\right\}_{i=1}^{\infty}$ and $\left\{M_i\right\}_{i=1}^{\infty}$. The intersection $\bigcap_{i} \overline{U_{i}}$ contains a point $z$ which must satisfy that
\begin{equation*}
\liminf_n
\frac{1}{n}\sum_{j=1}^{n} F(\phi^{-j}(z)) \ \leq \ 0  
\end{equation*}
and
\begin{equation*}
\liminf_n\frac{1}{n} \sum_{i=0}^{n-1} -F(\phi^{i}(z))  \ \leq \ 0 \ .
\end{equation*}
It follows therefore from Theorem \ref{thm31} that there exists an $e^{-F}$-conformal measure.

\end{proof}

\section{Variations with $\beta$}

In this section we describe two examples to show how the structure of the collection of $e^{\beta F}$-conformal measures can vary with $\beta$.

\subsection{Variation of ergodic conformal measures of type $I_{\infty}$} \label{section7}

An atomic $e^{\beta F}$-conformal measure concentrated on a single infinity $\phi$-orbit is of type $I_{\infty}$. The following result shows that such $e^{\beta F}$-conformal measures occur in abundance for any non-periodic homeomorphism provided the potential $F$ is well-chosen. 

\begin{thm}\label{thm73}
Let $X$ be a compact metric space and $\phi : X \to X$ a homeomorphism. Let $\{x_{1},x_{2}, \dots, x_{q}\} \subseteq X$ be a finite subset consisting of points $x_p$ in $X$ that are not periodic under $\phi$ and have disjoint orbits, and choose a $\beta_{p}  \ < \ 0$ and an interval $J_{p}$ of the form $]-\infty, \beta_{p}]$ or $]-\infty, \beta_{p}[$ for each $1\leq p \leq q$. There exists a continuous function $F:X\to \mathbb{R}$ such that there is an $e^{\beta F}$-conformal measure concentrated on the $\phi$-orbit of $x_p$ if and only if
\begin{equation*}
\beta \in J_{p} 
\end{equation*}
for $p = 1,2,\cdots , q$. The potential $F$ can be chosen such that the flow $\alpha^{F}$ is approximately inner.
\end{thm}

The proof of Theorem \ref{thm73} is an explicit construction which is rather long and technical and it has been relegated to Appendix \ref{AppA}.

\subsection{Variation of the factor type} 

In none of the examples of minimal homeomorphisms we have considered so far, have we seen the factor type of the $e^{\beta F}$-conformal measures vary with $\beta$. In this section we describe how an example by Baggett, Medina and Merrill constructed for the proof of Theorem 3 in \cite{BMM} can be reworked to show that such a variation can occur for any irrational rotation of the circle. Neither the statement nor the proof of Theorem 3 in \cite{BMM} is quite enough for our purposes, but a careful elaboration of the construction in \cite{BMM} can be performed to yield the following

\begin{prop}\label{prop74} 
Let $\phi : \mathbb T \to \mathbb T$ be an irrational rotation of the circle $\mathbb T$ and let $\lambda$ be Lebesgue measure of $\mathbb T$. There is a continuous real-valued function $F : \mathbb T \to \mathbb R$, a non-negative Borel function $v: \mathbb T \to \mathbb R$ and an increasing sequence $\{n_k\}_{k=1}^{\infty}$ in $\mathbb N$ such that
\begin{itemize}
\item[i)] $F(x) = v(x) - v (\phi(x))$ for $\lambda$-almost all $x$, 
\item[ii)] $\int_{\mathbb T} F \ \mathrm{d} \lambda \ = \ 0$,
\item[iii)] $v \notin L^1(\lambda)$, and
\item[iv)] $\left|\sum_{i=0}^{n_k-1} F\left(\phi^i(x)\right)\right| \leq 2$ for all $x \in \mathbb T$ and all $k$.
\end{itemize}
\end{prop}

The last condition iv), which is crucial for our purpose, does not follow for free from the construction in \cite{BMM} so we give the details in Appendix \ref{AppB}. Here we just point out the following

\begin{prop}\label{prop74(2)} 
Let $\phi : \mathbb T \to \mathbb T$ be an irrational rotation of the circle $\mathbb T$. There is a continuous real-valued function $F : \mathbb T \to \mathbb R$ such that
\begin{itemize}
\item There are $e^{\beta F}$-conformal measures for all $\beta \in \mathbb R$.
\item For $\beta \geq 0$ there is exactly one ergodic $e^{\beta F}$-conformal measure of type $II_1$.
\item For $\beta < 0$ all ergodic $e^{\beta F}$-conformal measures are of type $II_{\infty}$ or type $III$.
\end{itemize}
\end{prop}
\begin{proof} It follows from condition ii) in Proposition \ref{prop74} and Corollary \ref{cor37} that there are $e^{\beta F}$-conformal measures for all $\beta \in \mathbb R$. It follows from condition iv) in Proposition \ref{prop74} that the sum \eqref{eq73} diverges for all $\beta \in \mathbb R$ and all $x \in \mathbb T$. There are therefore no atomic $e^{\beta F}$-conformal measures for any $\beta \in \mathbb R$, and hence all ergodic $e^{\beta F}$-conformal measures are of type $II_1$, $II_{\infty}$ or $III$. Since $v$ is non-negative and Borel, $e^{-\beta v} \in L^1(\lambda)$ for all $\beta \geq 0$. It follows therefore from condition i) in Proposition \ref{prop74} that
$$
\mathrm{d}m = e^{-\beta v} \mathrm{d}\lambda
$$
is an $e^{\beta F}$-conformal measure $m$, clearly of type $II_1$, for all $\beta \geq 0$. For each such $\beta$ it is the only  $e^{\beta F}$-conformal measure of type $II_1$ because $\lambda$ is the only $\phi$-invariant Borel probability measure. It remains to show that there are no $e^{\beta F}$-conformal measures of type $II_1$ when $\beta < 0$. Assume for a contradiction that there is such a measure. It is concentrated on the set \eqref{eqadd3} and it follows therefore from Theorem 1.4 of \cite{GS} that there is also an ergodic $e^{\beta F}$-conformal measure $m$ of type $II_1$. Then $m$ is equivalent to $\lambda$ and the Radon-Nikodym derivative
$$
h \ = \ \frac{ \mathrm{d} m}{\mathrm{d} \lambda}
$$
is in $L^1(\lambda)$ and strictly positive $\lambda$-almost everywhere. Set $u = \log h$. By using that $m$ is $e^{\beta F}$-conformal we find that $e^{\beta F(x)}h(x) = h\circ \phi(x)$ for $\lambda$-almost all $x \in \mathbb T$, or $\beta F(x) = u(  \phi(x)) - u(x)$ for $\lambda$-almost all $x$ and $e^{u} \in L^1(\lambda)$. It follows that $(u + \beta v) \circ \phi(x) = (u + \beta v)(x)$ for $\lambda$-almost all $x$, and hence by ergodicity that there is a constant $t \in \mathbb R$ such that $-\beta v(x) = u(x) +t$ for $\lambda$-almost all $x$. This implies that
$$
\int_{\mathbb T} e^{-\beta v} \ \mathrm{d}\lambda \ = \ e^t \int_{\mathbb T} e^{ u} \ \mathrm{d}\lambda  \ < \ \infty \ .
$$
However, since $v \notin L^1(\lambda)$ by condition iii) in Proposition \ref{prop74} it follows from Jensen's inequality that
$$
\int_{\mathbb T} e^{-\beta v} \ \mathrm{d}\lambda \ \geq \ \exp \left( \int_{\mathbb T} - \beta v \ \mathrm{d} \lambda \right) \ = \ \infty \ .
$$
This contradiction shows that all ergodic $e^{\beta F}$-conformal measures are of type $II_{\infty}$ or type $III$ when $\beta < 0$.

\end{proof}

We note that Proposition \ref{prop74(2)} does not reveal the complete picture about the $e^{\beta F}$-conformal measures for the triple $(X,\phi,F)$. In particular, we do not know how many $e^{\beta F}$-conformal measures there are when $\beta \neq 0$. 

\section{Appendix A: Proof of Theorem \ref{thm73}}\label{AppA}
We assume throughout this appendix that we have a finite non-empty set $P:=\{x_1,x_2, \cdots, x_{q}\}$ of points in $X$ that are not periodic and have disjoint orbits under $\phi$. 
We choose for each point $x_{p}\in P$ two sequences $\{a_{i}^{p}\}_{i=0}^{\infty}$ and $\{b_{i}^{p}\}_{i=0}^{\infty}$ of positive real numbers satisfying that:
\begin{enumerate}
\item  $\lim_{i\to \infty} b_{i}^{p}=0$,
\item  $\sum_{i=0}^{\infty} a_{i}^{p} \ = \ \sum_{i=0}^{\infty} b_{i}^{p} \ = \ \infty$, and
\item $0  <  a_{i}^{p}  \leq  b_{i}^{p}$ for each $i\in \mathbb{N}\cup\{0\}$ and $\sum_{i=0}^{\infty} (b_{i}^{p}-a_{i}^{p}) \ = \ \infty$.
\end{enumerate}
Set $S_{i}:=\max_{p } \{b_{i}^{p}\}$ for each $i\in \mathbb{N}\cup\{0\}$. Then $\{S_{i} \}_{i=0}^{\infty}$ is a sequence of positive numbers that converges to $0$. 

\begin{lemma}\label{lemteknisk}
There exists an increasing sequence of natural numbers $\{N_i\}_{i=0}^{\infty}$ such that $N_n \geq 3(n+1)$ for all $n$, and for all $p \in \{1,2,3,\cdots ,q\}$, 
\begin{enumerate}
\item[$(a_{n})$] $b_{j}^{p}<2^{-n}$ when $j\geq N_{n}-1$, and

\smallskip
\item[$(b_{n})$]
$
\sum_{k=n+1}^{j-1} b_{k}^{p} \ \geq \ \sum_{k=0}^{j-1} a_{k}^{p}+\sum_{i=0}^{n} (2i+1)S_{i} \  $
when $j \geq N_n$.
\end{enumerate}
\end{lemma}
\begin{proof} This follows easily from (1) and (3).
\end{proof}


To construct $F$ we will first recursively construct a sequence of functions $\{f_{n}\}_{n=0}^{\infty}\subseteq C(X)$ and then define $F$ as their sum. For this purpose we construct for each $n = 0,1,2,3, \cdots$ and each $p \in \{1,2,\cdots, q\}$ an open neighbourhood $V^p_n$ of $\phi^n(x_p)$ such that
\begin{itemize}
\item[($1_n$)] $V^p_n \cap \phi^{-2n-1}\left(V^p_n\right) = \emptyset$  .
\end{itemize}
Set
$$
W^p_n = V^p_n \cup \phi^{-2n-1}\left(V^p_n\right) \ .
$$
We will arrange that
\begin{itemize}
\item[($2_n$)] $W^p_n \cap W^{p'}_n = \emptyset$ when $p \neq p'$, and
\item[($3_n$)] if $\phi^{i}(x_{l} ) \in \overline{W_{n}^{p}} \backslash \{\phi^{n}(x_{p}), \phi^{-n-1}(x_{p})\}$ for any $l$ then $|i|>N_{n}$,
\end{itemize}
where $\{N_n\}_{n=0}^{\infty}$ is the sequence from Lemma \ref{lemteknisk}. As a final requirement we arrange that
\begin{itemize}
\item[($4_n$)] if $i \in \{0,1,\cdots, n-1\}$ and $\{\phi^n(x_p), \phi^{-n-1}(x_p)\} \cap\bigcup_{l=1}^q \overline{W^l_i} = \emptyset$, then $W^p_n \cap \bigcup_{p=1}^q \overline{W^p_i} = \emptyset$ . 
\end{itemize}
Since condition $(4_0)$ is vacuous it is easy to construct the sets $V^p_0, \ 1\leq p \leq q$. Assume that we have constructed the sets $V^p_i, 1 \leq p \leq q, \ 0 \leq i \leq n-1$. Pick a family of mutually disjoint open sets $U_{+}^{p}, U_{-}^{p}, p = 1,\cdots ,q$, such that $\phi^{n}(x_{p})\in U_{+}^{p}$ and $\phi^{-n-1}(x_{p})\in U_{-}^{p}$ for each $x_{p}\in P$. For each $p \in \{1, \dots , q\}$, set
\begin{equation*}
M_{n}^{p}:=\left\{ \phi^{i}(x_{l}) \ | \ |i|\leq N_{n} +2n+1, \ l\in \{1, \dots q\} \right\}\setminus\{\phi^{n}(x_{p}), \phi^{-n-1}(x_{p})\}
\end{equation*}
and 
\begin{align*}
K_{n}^{p}  \ = \bigcup_{i,l} \overline{W^l_i}
\end{align*}
where we take the union over all $(i,l)$ such that $0 \leq i \leq n-1$ and $\{\phi^{n}(x_{p}), \phi^{-n-1}(x_{p})\} \cap \overline{W^l_i} = \emptyset$. Choose $V_{n}^{p}$ open with $\phi^{n}(x_{p})\in V_{n}^{p}$ such that
\begin{equation*}
V_{n}^{p}\subseteq \overline{V_{n}^{p}}\subseteq \left(\left(U_{+}^{p} \backslash  K^p_n \right) \cap \phi^{2n+1}\left( U_{-}^{p} \backslash K^p_n\right)\right)  \backslash M_n^p \ .
\end{equation*}
With this choice $(1_n)-(4_n)$ all hold and we have constructed the desired sequence of open sets. For each $p$ and $n$ we choose a function $g_{n}^{p} \in C(X)$ with $\supp g_{n}^p  \subseteq V_{n}^{p}$, $0 \leq g_{n}^{p}\leq b_{n}^{p}$ and $g_{n}^{p}(\phi^{n}(x_{p}))=b_{n}^{p}$, and set
$$
f_{n}^{p}\ = \ g_{n}^{p} \ - \ g_{n}^{p} \circ \phi^{2n+1} \ .
$$
Then $\supp f_{n}^{p} \subseteq W_{n}^{p}$ and $f_{n}^{p} \in C(X)$. Set
\begin{equation*}
f_{n}=\sum_{p=1}^{q} f_{n}^{p} \in C(X) \ .
\end{equation*}
Since $\supp f_{n}^{p} \subseteq W_{n}^{p}$ and $W_{n}^{p}\cap W_{n}^{l}=\emptyset$ for each $p\neq l$ we see from the construction that 
\begin{equation}\label{Remark 82}
\lVert f_{n} \rVert_{\infty} = \max_{p } \lVert f_{n}^{p} \rVert_{\infty} =S_{n} \ .
\end{equation}

\begin{lemma}
The sum $F = \sum_{n=0}^{\infty} f_n$ converges uniformly to a continuous function $F:X\to \mathbb{R}$.
\end{lemma}

\begin{proof} Set $F_m = \sum_{n=0}^{m} f_n$. Let $\varepsilon>0$ be given, and choose $N$ such that
\begin{equation*}
S_{i} + \sum_{k=i}^{\infty} 2^{-k} < \varepsilon
\end{equation*}
for all $i \geq N$. We show that $\left|F_m(x) - F_N(x)\right| \leq \epsilon$ for all $m\geq N$ and all $x \in X$, which will prove the lemma.  Fix $x\in \supp \left(F_m - F_N\right)$. Then
\begin{equation*}
F_{m}(x)- F_{N}(x) =\sum_{n=N+1}^{m} f_{n}(x) \ .
\end{equation*}
Let $i_{1}<i_{2}<i_{3}<\dots < i_{l}$ be the numbers between $N+1$ and $m$ for which $x\in \supp f_{i_{j}}$. If $l = 1$ it follows from \eqref{Remark 82} that $\left|F_m(x) - F_N(x)\right| \leq S_{i_1} \leq \epsilon$.  Assume $l > 1$ and consider $1 \leq n < l$.  Since $x\in \supp f_{i_{n}}$ and $x\in \supp f_{i_{n+1}}$ there is a $p \in \{1, \dots , q\}$ such that $x\in W_{i_{n+1}}^{p}\cap \supp f_{i_{n}}$.  It follows from $(4_{i_{n+1}})$ that $\{\phi^{i_{n+1}}(x_{p}), \phi^{-i_{n+1}-1}(x_{p}) \} \cap  \overline{W_{i_{n}}^s}\neq \emptyset$ for some $s \in \{1,2,\cdots , q\}$ and then from $(3_{i_n})$ that $i_{n+1}\geq N_{i_{n}}$. Then $b_{i_{n+1}}^{r} \leq 2^{-i_{n}}$ for all $r\in \{1, \dots , q\}$ by $(a_{i_n})$ of Lemma \ref{lemteknisk}.  In combination with \eqref{Remark 82} this shows that
\begin{equation*}
\lvert F_{m}(x)- F_{N}(x) \rvert \ \leq  \ \sum_{n=N+1}^{m} \lvert f_{n}(x)\rvert \
\leq \ \sum_{j=1}^{l} \lvert f_{i_{j}}(x)\rvert
\ \leq \ S_{i_{1}} +\sum_{j=2}^{l} 2^{-i_{j-1}} \ \leq  \ \varepsilon \ .
\end{equation*}

\end{proof}

\begin{remark}
Notice that if $\nu$ is a $\phi$-invariant measure then for any $p$ and $n$,
$$
\int f_{n}^{p} \ \mathrm{d}\nu 
 \ = \ \int g_{n}^{p} \ \mathrm{d}\nu - \int g_{n}^{p} \circ \phi^{2n+1} \ \mathrm{d}\nu \ = \ 0 \ ,
$$
and hence $\int F \ \mathrm{d}\nu =0$. By Theorem \ref{thm41} the corresponding flow $\alpha^{F}$ is approximately inner.
 \end{remark}

\begin{remark}\label{rem1} By construction the functions $g_{n}^{p} , \ g_{n}^{p}\circ \phi^{2n+1}, \ p=1, 2, \dots, q, $ have disjoint support. Hence, for all $n$ we have the implications
\begin{equation*}
g_n^p(x) \ \neq \ 0 \ \Rightarrow \  g_{n}^{l}\circ \phi^{2n+1}(x)  \ = \ g_{n}^{l}\circ\phi^{-2n-1}(x)  \ = \ 0 \ 
\end{equation*}
for all $l \in\{ 1,2,\cdots, q\}$, and  
\begin{equation*}
g_n^p(x) \ \neq \ 0 \ \Rightarrow \  g_{n}^{l}(x) \ = \ 0 \ 
\end{equation*}
when $l \in\{ 1,2,\cdots, q\}\setminus \{p\}$.
\end{remark}

\begin{lemma} \label{lulig1}
Let $n\in \mathbb{N}\cup \{0\}$ and $ p \in \{1, \dots , q\}$. Then
\begin{equation*}
\sum_{k=n}^{m} f_{n}(\phi^{k}(x_{p})) \in [-(2n+1)S_{n}, b_{n}^{p}] \  ,
\end{equation*}
for $m\geq n$, and if $\sum_{k=n}^{m} f_{n}(\phi^{k}(x_{p}))$ is different from $b_{n}^{p}$ then $m$ is so big that
\begin{equation*}
\sum_{j=0}^{m}a_{j}^{p}+\sum_{j=0}^{n}(2j+1)S_{j} \ \leq \ \sum_{j=n+1}^{m} b_{j}^{p} \ . 
\end{equation*}
\end{lemma}

\begin{proof}
Fix a $p \in \{1, \dots , q\}$. Since $g_{n}^{p}(\phi^{n}(x_{p}))=b_{n}^{p}$ it follows from Remark \ref{rem1} that $f_{n}(\phi^{n}(x_{p}))=b_{n}^{p}$. 

\smallskip

\emph{Claim 1: If $k>n$ and $f_{n}(\phi^{k}(x_{p})) \ > \ 0$, then $k\geq 3n+3$ and
$$
f_{n}(\phi^{k}(x_{p}))+f_{n}(\phi^{k-2n-1}(x_{p})) \ = \ 0 \ .
$$}

\smallskip

To prove Claim 1 note that if $f_{n}(\phi^{k}(x_{p}))>0$ then there exists some $l\in \{1, \dots , q\}$ such that $\phi^{k}(x_{p})\in W_{n}^{l}$ and hence $k>N_{n}\geq 3n+3$, using ($3_n$). Since $f_{n}(\phi^{k}(x_{p}))>0$ implies $g_{n}^{l}(\phi^{k}(x_{p})) \neq 0$, it follows from Remark \ref{rem1} that $g_{n}^{r}(\phi^{k+2n+1}(x_{p})) = g_{n}^{r}(\phi^{k-2n-1}(x_{p})) =0$ for all $r$ while $ g_{n}^{r}(\phi^{k}(x_{p})) = 0$ for $r\neq l$. Thus
\begin{align*}
&f_{n}(\phi^{k}(x_{p}))+f_{n}(\phi^{k-2n-1}(x_{p})) \\
&= \sum_{r=1}^{q} \big(g_{n}^{r}(\phi^{k}(x_{p}))-g_{n}^{r}(\phi^{k+2n+1}(x_{p}))\big)+\sum_{r=1}^{q}\big(g_{n}^{r}(\phi^{k-2n-1}(x_{p})) - g_{n}^{r}(\phi^{k}(x_{p})) \big) \\
&= g_{n}^{l}(\phi^{k}(x_{p}))-g_{n}^{l}(\phi^{k}(x_{p}))=0 \ .
\end{align*}

\smallskip

\emph{Claim 2: If $k>n$ and $f_{n}(\phi^{k}(x_{p}))<0$ then:
\begin{equation*}
f_{n}(\phi^{k}(x_{p}))+f_{n}(\phi^{k+2n+1}(x_{p}))=0 \ .
\end{equation*}}

\smallskip

To prove Claim 2 note that if $f_{n}(\phi^{k}(x_{p}))<0$ then there exists some $l\in \{1, \dots , q\}$ such that $\phi^{k}(x_{p})\in W_{n}^{l}$, and in particular $g_{n}^{l}(\phi^{k+2n+1}(x_{p}))\neq 0$, so Remark \ref{rem1} implies that $g_{n}^{r}(\phi^{k}(x_{p}))=0$ and $g_{n}^{r}(\phi^{4n+2}(\phi^{k}(x_{p})))=0$ for all $r$ and that $g_{n}^{r}(\phi^{k+2n+1}(x_{p}))= 0$ for all $r \neq l$. Hence 
\begin{align*}
&f_{n}(\phi^{k}(x_{p}))+f_{n}(\phi^{k+2n+1}(x_{p})) \\
&= \ \sum_{r=1}^{q} \big(g_{n}^{r}(\phi^{k}(x_{p}))-g_{n}^{r}(\phi^{k+2n+1}(x_{p}))\big) \\
& \ \ \ \ \ \ \ \ \ \ \ \ \ \ + \ \sum_{r=1}^{q}\big(g_{n}^{r}(\phi^{k+2n+1}(x_{p})) - g_{n}^{r}(\phi^{k+4n+2}(x_{p})) \big) \\
&= \ -g_{n}^{l}(\phi^{k+2n+1}(x_{p}))+g_{n}^{l}(\phi^{k+2n+1}(x_{p}))=0 \ .
\end{align*}

\smallskip

Using the two claims we can complete the proof as follows. Let $A_+=\{ n+1\leq k \leq m : \   f_{n}(\phi^{k}(x_{p}))>0\}$ and $A_-=\{ n+1\leq k \leq m : \  f_{n}(\phi^{k}(x_{p}))<0\}$. Then
\begin{equation*}
\sum_{k=n+1}^{m} f_{n}(\phi^{k}(x_{p})) = \sum_{k\in A_+} f_{n}(\phi^{k}(x_{p}))+\sum_{k\in A_-} f_{n}(\phi^{k}(x_{p})) \ .
\end{equation*}
By the first claim the map $\eta(k) = k-2n-1$ takes $A_+$ into $A_-$ while the second claim implies that $A_- \cap[0,m-2n-1] \subseteq \eta(A_+)$. It follows therefore from Claim 1 that
\begin{equation*}
 \sum_{k\in A_+} f_{n}(\phi^{k}(x_{p}))+\sum_{k\in A_-} f_{n}(\phi^{k}(x_{p})) \ = \ \sum_{k\in A_-\setminus \eta(A_+)} f_{n}(\phi^{k}(x_{p})) \ .
\end{equation*}
Since $A_-\setminus \eta(A_+) \subseteq A_-\cap [m-2n, m]$ and since $-S_{n} \leq f_{n}(\phi^{k}(x_{p}))\leq 0$ for $k\in A_-$ we find that 
\begin{equation*}
\sum_{k=n}^{m} f_{n}(\phi^{k}(x_{p})) \ = \ b_{n}^{p} \ +\sum_{k\in A_-\setminus \eta(A_+)} f_{n}(\phi^{k}(x_{p}))\ \in \ [-(2n+1)S_{n}, b_{n}^{p}] \ .
\end{equation*}
For the last statement in the lemma let $j> n$ be the smallest number with $f_{n}(\phi^{j}(x_{p}))\neq 0$. If $m<j$ then $\sum_{k=n}^{m} f_{n}(\phi^{k}(x_{p}))=b_{n}^{p}$, so we must have $m\geq j$. But $f_{n}(\phi^{j}(x_{p}))\neq 0$ implies $\phi^{j}(x_{p})\in W_{n}^{l}$ for some $l\in \{1, \dots , q\}$ and hence $j > N_{n}$, using ($3_n$). The stated inequality follows therefore from $(b_{n})$ in Lemma \ref{lemteknisk}.
\end{proof}

The proof of the following Lemma is very similar to that of Lemma \ref{lulig1}, and we leave it to the reader.

\begin{lemma} \label{lulig2}
Let $n\in \mathbb{N}\cup \{0\}$ and $p\in \{1, \dots , q\}$. For any $m\geq n$ then:
\begin{equation*}
\sum_{k=n+1}^{m} f_{n}(\phi^{-k}(x_{p})) \in [-b_{n}^{p}, (2n+1)S_{n}],
\end{equation*}
and if $\sum_{k=n+1}^{m} f_{n}(\phi^{-k}(x_{p})) \neq  -b_{n}^{p}$ then $m$ is so big that
\begin{equation*}
 \sum_{j=0}^{m-1} a_{j}^{p} +
\sum_{j=0}^{n}(2j+1)S_{j} \ \leq \ \sum_{j=n+1}^{m-1} b_{j}^{p} \ .
\end{equation*}
\end{lemma}

\begin{lemma}\label{p33} For $p\in \{1, \dots , q\}$ and $m\in \mathbb{N}$,
\begin{equation} \label{epos}
\sum_{i=0}^{m} a_{i}^{p}\ \leq \ \sum_{i=0}^{m} F(\phi^{i}(x_{p})) \ \leq \  \sum_{i=0}^{m} b_{i}^{p} 
\end{equation}
and
\begin{equation}\label{eneg}
-\sum_{i=0}^{m-1} b_{i}^{p} \ \leq \  \sum_{i=1}^{m} F(\phi^{-i}(x_{p})) \ \leq \ -\sum_{i=0}^{m-1 }a_{i}^{p} \ .
\end{equation}
\end{lemma}

\begin{proof}
Fix a $p\in \{1, \dots, q\}$. For the first statement, notice that $f_{j}(\phi^{i}(x_{p}))=0$ when $0\leq i<j$ by construction. This implies that
\begin{align*}
&\sum_{i=0}^{m} F(\phi^{i}(x_{p}))=
\sum_{i=0}^{m} \sum_{j=0}^{i} f_{j}(\phi^{i}(x_{p})) \\
&=\sum_{j=0}^{m}f_{0}(\phi^{j}(x_{p}))+ \sum_{j=1}^{m}f_{1}(\phi^{j}(x_{p}))+\cdots+\sum_{j=m}^{m}f_{m}(\phi^{j}(x_{p})) \ .
\end{align*}
The upper bound in \eqref{epos} now follows directly from Lemma \ref{lulig1}. If $\sum_{j=i}^{m}f_{i}(\phi^{j}(x_{p}))=b_{i}^{p}$ for each $0\leq i \leq m$ the lower bound is trivial, so assume this is not the case and choose $0\leq n \leq m$ largest such that $\sum_{j=n}^{m}f_{n}(\phi^{j}(x_{p}))\neq b_{n}^{p}$. It now follows from Lemma \ref{lulig1} that that $m$ is so big that
\begin{equation*}
\sum_{j=n+1}^{m} b_{j}^{p} \geq \sum_{j=0}^{m} a_{j}^{p}+\sum_{j=0}^{n} (2j+1) S_{j}, 
\end{equation*}
and since $n$ was chosen largest then Lemma \ref{lulig1} gives
\begin{align*}
&\sum_{i=0}^{m} F(\phi^{i}(x_{p}))=\sum_{k=0}^{m}\sum_{j=k}^{m}f_{k}(\phi^{j}(x_{p}))
=\sum_{k=0}^{n}\sum_{j=k}^{m}f_{k}(\phi^{j}(x_{p}))+\sum_{k=n+1}^{m}b_{k}^{p} \\
& \geq \ \sum_{k=0}^{n} -(2k+1) S_{k}+\sum_{j=0}^{m} a_{j}^{p}+\sum_{j=0}^{n} (2j+1) S_{j} \ = \ 
\sum_{j=0}^{m} a_{j}^{p} \ .
\end{align*}
To prove \eqref{eneg} note that $f_{j}(\phi^{-i}(x_{p}))=0$ when $j\geq i$ so that 
\begin{align*}
&\sum_{i=1}^{m} F(\phi^{-i}(x_{p}))=
\sum_{i=1}^{m} \sum_{j=0}^{i-1} f_{j}(\phi^{-i}(x_{p})) \\
&=\sum_{j=1}^{m}f_{0}(\phi^{-j}(x_{p}))+ \sum_{j=2}^{m}f_{1}(\phi^{-j}(x_{p}))+ \ \cdots \ +\sum_{j=m}^{m}f_{m-1}(\phi^{-j}(x_{p})) \ .
\end{align*}
This yields the lower bound by Lemma \ref{lulig2}, and to obtain the upper bound assume $0\leq n \leq m-1$ is chosen largest with $\sum_{k=n+1}^{m}f_{n}(\phi^{-k}(x_{p}))\neq -b_{n}^{p}$; if such a $n$ does not exist the stated upper bound is trivially true. Using Lemma \ref{lulig2} we get
\begin{align*}
&\sum_{i=1}^{m} F(\phi^{-i}(x_{p}))=
\sum_{k=0}^{m-1} \sum_{j=k+1}^{m} f_{k}(\phi^{-j}(x_{p}))
=\sum_{k=0}^{n} \sum_{j=k+1}^{m} f_{k}(\phi^{-j}(x_{p}))-\sum_{k=n+1}^{m-1} b_{k}^{p} \\
&\leq \ \sum_{k=0}^{n} (2k+1)S_{k} - \sum_{j=0}^{m-1} a_{j}^{p}- \sum_{j=0}^{n} (2j+1) S_{j}   \ = \   - \sum_{j=0}^{m-1} a_{j}^{p} \ .
\end{align*}
\end{proof}

We can now complete the proof of Theorem \ref{thm73}: Partition $\{1, \dots, q\}=\mathcal{C} \sqcup \mathcal{O}$ such that $I_{p}=\ ]-\infty, \beta_{p}]$ when $p\in \mathcal{C}$ and $I_{p}= \ ]-\infty, \beta_{p}[$ when $p\in \mathcal{O}$. Choose strictly decreasing sequences $\{c_n\}_{n=0}^{\infty}$ and $\{d_n\}_{n=0}^{\infty}$ of positive numbers with $c_{n}/c_{n+1}\to 1$, $c_{0}=1$ and such that $\sum_{n=0}^{\infty} c_n^s < \infty$ when $s \geq 1$, $\sum_{n=0}^{\infty} c_n^s = \infty$ when $s < 1$, $\sum_{n=0}^{\infty} d_n^s < \infty$ when $s > 1$, $\sum_{n=0}^{\infty} d_n^s = \infty$ when $s \leq 1$. Consider a $p \in \mathcal C$. Set
$$
a^p_n \ = \ \frac{\log(c_{n+1}) - \log(c_n)}{\beta_p} \ .
$$
Find a non-increasing sequence $\{t_{n}^{p}\}_{n=0}^{\infty}\ \subseteq  \ ]1,2[$ such that $\lim_{n \to \infty} t_{n}^{p} =1$ and $b_{n}^{p}:=t_{n}^{p}a_{n}^{p}$ satisfies the criteria (1), (2) and (3) from the beginning of this Appendix. When $\beta \geq 0$ it follows from the first inequality of\eqref{epos} that the sum from 3) in Lemma \ref{25-10-19} is divergent and hence that there is no $e^{\beta F}$-conformal measure concentrated on the orbit of $x_p$. Let $\beta \ \in \ ]\beta_{p}, 0[$. Choose $\beta/\beta_{p}<s<1$. By the choice of $\{t_{n}^{p}\}$ there is a $N\in \mathbb{N}$ such that
\begin{equation*}
s\sum_{i=0}^{k-1} b_{i}^{p} \ \leq  \ \sum_{i=0}^{k-1} a_{i}^{p}
\end{equation*}
for  all $k\geq N$. Using \eqref{epos} we find that
\begin{equation*}
\exp\big(\beta \sum_{i=0}^{k-1}F(\phi^{i}(x_{p}))\big) \  
\geq \ \exp\big(\beta \sum_{i=0}^{k-1} b_{i}^{p}   \big)
\ \geq \  \exp\big( s^{-1}\beta \sum_{i=0}^{k-1} a_{i}^{p}   \big) \ 
= \ c_{k}^{\frac{s^{-1}\beta}{\beta_{p}}} \ 
\end{equation*}
for $k\geq N$. Since $\frac{s^{-1}\beta}{\beta_{p}} <1$ it follows again from Lemma \ref{25-10-19} and the properties of $\{c_n\}$ that there is no $e^{\beta F}$-conformal measure concentrated on $x_{p}$. When $\beta \leq \beta_p$ we find that
\begin{equation*}
\exp\big(\beta \sum_{i=0}^{k-1}F(\phi^{i}(x_{p}))\big) 
\leq \exp\big(\beta \sum_{i=0}^{k-1} a_{i}^{p}   \big)
=  c_{k}^{\beta/\beta_{p}}
\end{equation*} 
and 
\begin{equation*}
\exp\big(-\beta \sum_{i=1}^{k}F(\phi^{-i}(x_{p}))\big) 
\leq \exp\big(\beta \sum_{i=0}^{k-1} a_{i}^{p}   \big)
=  c_{k}^{\beta/\beta_{p}} \ ,
\end{equation*} 
and now 3) of Lemma \ref{25-10-19} implies that there is a $e^{\beta F}$-conformal measure concentrated on the orbit of $x_{p}$. The elements of $\mathcal O$ are handled in the same way, using the sequence $\{d_n\}$ instead.

\qed

\section{Appendix B: Proof of Proposition \ref{prop74}}\label{AppB}

We modify the construction used to prove Theorem 3 in \cite{BMM}. Throughout the appendix $\alpha$ will be an arbitrary but fixed irrational number in $[0,1]$. We start by recalling Dirichlet's theorem for approximation of irrational numbers.

\begin{thm}\label{Dir} (Dirichlet)
The inequality
\begin{equation*}
\big\lvert \alpha - \frac{p}{q} \big\rvert \leq \frac{1}{q^{2}}
\end{equation*}
is satisfied for infinitely many integers $p$ and $q$. 
\end{thm}

When $r \geq 0$ we use the notation $[r]$ for the unique number in $[0,1[$ with $r-[r]\in \mathbb{N}$. We choose first a rational number $p_{1}/q_{1} \in \mathbb{Q}$ such that
\begin{equation*}
\big\lvert \alpha - \frac{p_{1}}{q_{1}} \big\rvert \ \leq \  \frac{1}{1^{3} 2^{2+1}q_{1}} \ .
\end{equation*}
This is possible by Theorem \ref{Dir}; just take $q_{1} \geq 1^{3} 2^{2+1}$. Now define
\begin{equation*}
u_{1}(x)=
\begin{cases}
2^{1+1}+1\cdot 2^{2+1}q_{1} x & \text{ if } x\in \big[-\frac{1}{1\cdot2^{1} q_{1}} , 0 \big ] \  , \\ 
2^{1+1}-1\cdot 2^{2+1}q_{1} x & \text{ if } x\in \big[0, \frac{1}{1\cdot2^{1} q_{1}} \big ] \ , \\
 0 \ , \ \text{otherwise} \ ,
\end{cases}
\end{equation*}
and consider the sawtooth function $\omega_{1}$ defined 
\begin{equation*}
\omega_{1} (x) : = \sum_{p=0}^{q_{1}} u_{1}(x-p/q_{1})
\end{equation*}
when $x \in [0,1]$. Since $\omega_1(0) = \omega_1(1)$ we can and will consider $\omega_1$ as a $1$-periodic function on $\mathbb R$. Then $\omega_1$ is in fact $\frac{1}{q_1}$-periodic. Since $\omega_{1}$ is a uniformly continuous function there exists a number $0 < \varepsilon_{1} < 1$ such that if $\lvert x-y\rvert \leq \varepsilon_{1} $ then $\lvert \omega_{1}(x)-\omega_{1}(y)\rvert \leq 1/1^{2}$. Now choose a integer $n_{2} \in \mathbb{N}$ with $n_{2} \geq 2$ such that $[n_{2} \alpha]\leq \varepsilon_{1}$; this is possible because of minimality of rotation by $\alpha$ since $(0, \varepsilon_{1})$ is open. Choose a rational number $p_{2}/q_{2}$ such that
\begin{equation}\label{e1}
\big\lvert \alpha - \frac{p_{2}}{q_{2}} \big\rvert \ \leq \  \frac{1}{2^{3} 2^{2\cdot 2+1} n_{2} q_{2}}=\frac{1}{2^{3} 2^{5} n_{2} q_{2}} \ ,
\end{equation}
which can be done by taking $q_{2} \ \geq  \ 2^{3} 2^{5} n_{2}$ in Theorem \ref{Dir}. Then
\begin{equation*}
\big\lvert \alpha - \frac{p_{2}}{q_{2}} \big\rvert \ \leq \  \frac{1}{2^{3} 2^{2\cdot 2+1} q_{2}}
\end{equation*}
since $n_2 \geq 1$, and it follows from \eqref{e1} that
\begin{equation*}
\big\lvert \alpha n_{2} - \frac{p_{2} n_{2}}{q_{2}} \big\rvert \ \leq \ \frac{1}{2^{3} 2^{2\cdot 2+1} q_{2}} \ \Rightarrow \  \big\lvert [\alpha n_{2}] - \frac{r(2)_{2}}{q_{2}} \big\rvert \ \leq \ \frac{1}{2^{3} 2^{2\cdot 2+1} q_{2}}
\end{equation*}
where $r(2)_{2}\in \mathbb{N}\cup\{0\}$ with $0\leq r(2)_{2} <q_{2}$. Now define
\begin{equation*}
u_{2}(x)=
\begin{cases}
2^{2+1}+2\cdot 2^{4+1}q_{2} x & \text{ if } x\in \big[-\frac{1}{2\cdot2^{2} q_{2}} , 0 \big ] \  , \\ 
2^{2+1}-2\cdot 2^{4+1}q_{2} x & \text{ if } x\in \big[0, \frac{1}{2\cdot2^{2} q_{2}} \big ]   \ , \\
0 \ , \ \text{otherwise} \ ,
\end{cases}
\end{equation*}
and the corresponding sawtooth function $\omega_{2}$ by
\begin{equation*}
\omega_{2} (x) : = \sum_{p=0}^{q_{2}} u_{2}(x-p/q_{2}) 
\end{equation*}
when $x \in [0,1]$, and extend it by periodicity to a $\frac{1}{q_2}$-periodic function on $\mathbb R$. Since $\omega_{2}$ is uniformly continuous there exists a number $0< \varepsilon_{2} \leq \varepsilon_{1}$ such that if $\lvert x-y\rvert \leq \varepsilon_{2} $ then $\lvert \omega_{2}(x)-\omega_{2}(y)\rvert \leq 1/2^{2}$. Now choose an integer $n_{3} \in \mathbb{N}$ with $n_{3} >n_{2}$ such that $[n_{3} \alpha]\leq \varepsilon_{2}$. As before Theorem \ref{Dir} gives a rational number $p_{3}/q_{3}$ such that
\begin{equation*}
\big\lvert \alpha - \frac{p_{3}}{q_{3}} \big\rvert \leq \frac{1}{3^{3} 2^{2\cdot 3+1} n_{3}n_{2} q_{3}}=\frac{1}{3^{3} 2^{7} n_{3}n_{2} q_{3}} \ ,
\end{equation*}
which implies that
\begin{equation*}
\big\lvert \alpha - \frac{p_{3}}{q_{3}} \big\rvert \ \leq \  \frac{1}{3^{3} 2^{2\cdot 3+1} q_{3}} \ .
\end{equation*}
Arguing as before there exists $0\leq r(3)_{2}, r(3)_{3} \leq q_{3}$ such that
\begin{equation*}
\big\lvert [\alpha n_{2}] - \frac{r(3)_{2}}{q_{3}} \big\rvert \ \leq \ \frac{1}{3^{3} 2^{2\cdot 3+1} q_{3}}
\end{equation*}
and
\begin{equation*}
\big\lvert [\alpha n_{3}] - \frac{r(3)_{3}}{q_{3}} \big\rvert \ \leq \ \frac{1}{3^{3} 2^{2\cdot 3+1} q_{3}} \ .
\end{equation*}
Now define
\begin{equation*}
u_{3}(x)=
\begin{cases}
2^{3+1}+3\cdot 2^{2\cdot 3+1}q_{3} x & \text{ if } x\in \big[-\frac{1}{3\cdot 2^{3} q_{3}} , 0 \big ] \  , \\ 
2^{3+1}-3\cdot 2^{2\cdot 3+1}q_{3} x & \text{ if } x\in \big[0, \frac{1}{3\cdot 2^{3} q_{3}} \big ]  \ , \\
0 \ , &  \ \text{otherwise} \ ,
\end{cases}
\end{equation*}
and then the sawtooth function $\omega_{3}$ as
\begin{equation*}
\omega_{3} (x) : = \sum_{p=0}^{q_{3}} u_{3}(x-p/q_{3})\ 
\end{equation*}
when $x \in [0,1]$ and extend to get a $\frac{1}{q_3}$-periodic function on $\mathbb R$. Since $\omega_{3}$ is uniformly continuous there exists a number $0< \varepsilon_{3} \leq \varepsilon_{2}$ such that if $\lvert x-y\rvert \leq \varepsilon_{3} $ then $\lvert \omega_{3}(x)-\omega_{3}(y)\rvert \leq 1/3^{2}$. Choose a integer $n_{4} \in \mathbb{N}$ with $n_{4}>n_{3}$ such that $[n_{4} \alpha]\leq \varepsilon_{3}$, etc.

\smallskip

We continue like this and get a sequence of continuous functions $\{\omega_{k}\}_{k=1}^{\infty}$ on $\mathbb R$, a sequence of rational numbers $\{p_{k}/q_{k}\}_{k=1}^{\infty}$, an increasing sequence of integers $\{n_{k}\}_{k=1}^{\infty}$ with $n_1 = 1$ and a decreasing sequence of positive real numbers $\{\varepsilon_{k}\}_{k=1}^{\infty}$ with $\varepsilon_1 <1$ such that for each $k\in \mathbb{N}$:
\begin{enumerate}
\item[a)] There are numbers $0\leq r(k)_{2},r(k)_{3}, \dots, r(k)_{k} \leq q_{k}$ such that
\begin{itemize}
\item $\big\lvert \alpha - \frac{p_{k}}{q_{k}} \big \rvert \  \leq \ \frac{1}{k^{3} 2^{2k+1} q_{k}}$ and  
\item $\big\lvert [n_{i}\alpha] - \frac{r(k)_{i}}{q_{k}} \big \rvert \ \leq \ \frac{1}{k^{3} 2^{2k+1} q_{k}}$
\end{itemize}
for $i=2, \dots , k$,
\item[b)] $\omega_{k}$ is non-negative with 
\begin{equation*}
\int_0^1 \omega_{k}(x) \ \mathrm{d}x \ = \ 2/k \ ,
\end{equation*}
\item[c)] $\omega_{k}(x+\frac{1}{q_k}) = \omega_k(x)$ for $x \in \mathbb R$ \ ,
\item[d)]  the support of $\omega_{k}$ in $[0,1]$ has Lebesgue measure $q_{k} \frac{2}{k2^{k} q_{k}}$ \ ,
\item[e)] $\lvert \omega_{k}(x)-\omega_{k}(y) \lvert  \ \leq \ k 2^{2k+1} q_k \left|x-y\right|$ for all $x,y \in \mathbb R$ \ , 
\item[f)] $\lvert \omega_{k}(x)-\omega_{k}(y) \lvert  \ \leq \ 1/k^{2}$ when $\lvert x-y\rvert \leq \varepsilon_{k}$ \ ,
\item[g)] $[n_{k+1}\alpha] \leq \varepsilon_{k}$ \ .
\end{enumerate}
The first condition in a) and b),c),d) and e) are all true for the construction in \cite{BMM} and have the following consequences, cf. \cite{BMM}:
\begin{itemize}
\item The sum $f(x) =\sum_{k=1}^{\infty} \omega_{k}(x)-\omega_{k}(x+\alpha)$ converges uniformly to a $1$-periodic continuous function on $\mathbb R$.
\item The sum $\omega(x) = \sum_{k=1}^{\infty} \omega_k(x)$ converges for Lebesgue almost all $x \in [0,1]$.
 \item $\int_0^1 \omega(x) \ \mathrm{d}x \ = \ \infty .$
\end{itemize}
We have added the second condition in a) and f) and g) in order to conclude
\begin{itemize}
\item $\sum_{l=1}^{\infty} \left|\omega_l(x) - \omega_l(x + n_k\alpha) \right| \leq 2$
\end{itemize}
for all $k > 1$ and all $x \in \mathbb R$. To see that this is true note that because the $\{\varepsilon_{n}\}$-sequence is decreasing it follows from g) that
\begin{equation*}
[n_{k}\alpha] \leq \varepsilon_{k-1} \leq \varepsilon_{l}
\end{equation*}
when $l < k$. Using c) and f) we find that
\begin{equation*}
\lvert \omega_{l}(x) - \omega_{l}(x+n_{k}\alpha) \rvert\leq \frac{1}{l^{2}}
\end{equation*}
for all $x \in \mathbb R$.
For $l\geq k$ we get from a) that 
\begin{equation*}
\big \lvert [n_{k}\alpha] - \frac{r(l)_{k}}{q_{l}}  \big\rvert\ \leq \ \frac{1}{l^{3}2^{2l+1}q_{l}}
\end{equation*}
and then from c), e) and a) that
\begin{equation*}
\begin{split}
&\lvert \omega_{l}(x) - \omega_{l}(x+n_{k}\alpha) \rvert \ = \ 
 \lvert \omega_{l}(x+\frac{r(l)_{k}}{q_{l}}) - \omega_{l}(x+n_{k}\alpha) \rvert \\
 &\leq
 \ l2^{2l+1} q_{l} \lvert [n_{k}\alpha]-   r(l)_{k}/q_{l} \rvert \ \leq \ \frac{1}{l^{2}}
\end{split}
\end{equation*}
for all $x \in \mathbb R$.
In conclusion we have for any $k>1$ and $x$ that
\begin{equation*}
\sum_{l=1}^{\infty} \lvert \omega_{l}(x) - \omega_{l}(x+n_{k}\alpha) \rvert \ \leq \ \sum_{l=1}^{\infty}\frac{1}{l^{2}} \ \leq \  2\ .
\end{equation*}
Note that this implies that for all $x$,
\begin{equation*}\label{15-03-19d}
\lvert \sum_{i=0}^{n_{k}-1} f(x+i\alpha) \rvert \ \leq \ 2 \ .
\end{equation*}

By identifying the circle $\mathbb T$ with $\mathbb R/ \mathbb Z$ the function $f$ becomes a potential $F : \mathbb T \to \mathbb R$. Putting $\omega(x) = 0$ when $\sum_{k=1}^{\infty} \omega_k(x) = \infty$ we can consider $\omega$ as a real-valued Borel function $v$ on $\mathbb T$. Then $F$ and $v$ have the properties specified in Proposition \ref{prop74}.

\end{document}